\documentclass{amsart}

\usepackage{amsmath,amsfonts,amssymb}
\usepackage{amsthm,upref}
\usepackage{enumerate}
\usepackage{cancel}
\usepackage[a4paper]{geometry}
\usepackage{graphicx}
\usepackage{subcaption,tabularx,float}
\usepackage{etoolbox}
\usepackage[final]{pdfpages}
\usepackage{import}
\usepackage{xifthen}
\usepackage{pdfpages}
\usepackage{transparent}
\usepackage{xparse}

\NewDocumentCommand\set{mg}{%
	\ensuremath{\bigl\{ #1 \IfNoValueTF{#2}{}{\bigm| #2} \bigr\}}%
}

\theoremstyle{plain}
\newtheorem{thm}{Theorem}[section]
\newtheorem{proposition}[thm]{Proposition}
\newtheorem{corollary}[thm]{Corollary}
\newtheorem{lemma}[thm]{Lemma}

\theoremstyle{definition}
\newtheorem{definition}[thm]{Definition}

\newtheorem{remark}[thm]{Remark}

\numberwithin{equation}{section}
\numberwithin{figure}{section}

\newcommand{\me}{\mathrm{e}}
\newcommand{\R}{\mathbb{R}}
\newcommand{\C}{\mathbb{C}}
\newcommand{\Q}{\mathbb{Q}}
\newcommand{\N}{\mathbb{N}}
\newcommand{\Z}{\mathbb{Z}}

\newcommand{\K}{\mathbb{K}}
\renewcommand{\H}{\mathcal{H}}
\newcommand{\nin}{\notin}

\newcommand{\bigslant}[2]{\left.\raisebox{.2em}{$#1$}\middle/\raisebox{-.2em}{$#2$}\right.}

\DeclareMathOperator{\Res}{Res}
\DeclareMathOperator{\Ima}{Im}

\DeclareMathOperator{\diag}{diag}

\DeclareMathOperator{\End}{End}
\DeclareMathOperator{\re}{Re}
\DeclareMathOperator{\coker}{coker}

\newcommand{\NHIM}{\mathcal{N}}

\patchcmd{\subsubsection}{-.5em}{.5em}{}{}
\patchcmd{\subsection}{-.5em}{.5em}{}{}

\title[Manifolds of Normally Hyperbolic Singularities]{Normal Forms for Manifolds of Normally Hyperbolic Singularities and Asymptotic Properties of Nearby Transitions}
\date{} 

\begin{document}
	\author{Nathan Duignan}
	\address[Nathan Duignan]{School of Mathematics and Statistics, University of Sydney, Camperdown, 2006 NSW, Australia}
	
	\begin{abstract}
		This paper contains theory on two related topics relevant to manifolds of normally hyperbolic singularities. First, theorems on the formal and $ C^k $ normal forms for these objects are proved. Then, the theorems are applied to give asymptotic properties of the transition map between sections transverse to the centre-stable and centre-unstable manifolds of some normally hyperbolic manifolds. A method is given for explicitly computing these so called Dulac maps. The Dulac map is revealed to have similar asymptotic structures as in the case of a saddle singularity in the plane.
	\end{abstract}

	\maketitle

	\section{Introduction}
			
	  Due to their persistence properties and common attributes with hyperbolic
	  singularities, normally hyperbolic manifolds have been studied and applied in
	  great depth by many authors, see for instance
	  \cite{wigginsNormallyHyperbolicInvariant1994}. However, there appears to be
	  little research aimed at normally hyperbolic manifolds consisting entirely of
	  singular points. This is primarily a consequence of their structural
	  instability under $ C^1 $-perturbations. Nevertheless, a general investigation
	  of these manifolds is warranted by recent applications in celestial mechanics
	  \cite{duignanC83regularisationSimultaneous2020,duignanChazyTypeAsymptoticsHyperbolic2020}, control theory
	  \cite{caillauSingularitiesMinTime2018}, regularisation of singularities
	  \cite{duignanRegularisationPlanarVector2019}, geometric singular perturbation
	  theory \cite{dumortierSmoothNormalLinearization2010}, and bifurcation theory
	  \cite{roussarieAlmostPlanarHomoclinic1996}.

		This work is a first venture into the properties of  normally hyperbolic manifolds of singularities considered in generality. Technical results on two related topics of normal form theory are provided. The first concerns normal form theory for these manifolds. This is studied in the formal, $ C^k $ and $C^\omega$ categories. The second is a study of transitions between sections transverse to the centre-stable and centre-unstable manifolds of normally hyperbolic manifolds consisting entirely of saddle singularities. We provide an extension of the work on hyperbolic saddles in $ \R^3 $ by Bonckaert and Naudot \cite{bonckaertAsymptoticPropertiesDulac2001}, and the `almost planar case' of Roussarie and Rousseau \cite{roussarieAlmostPlanarHomoclinic1996}. Moreover, the generalisation agrees with the particular application considered by Caillau et al. \cite{caillauSingularitiesMinTime2018}. The transition maps in the general case will be shown to share many properties of the well studied Dulac maps in the plane.
					
		The paper begins with an investigation of normal forms in Section \ref{sec:NormalForms}. In essence, normal form theory aims to define the ``simplest'' possible representation of a vector field $ X $. Two vector fields are said to be $ C^k $ (resp. analytically, formally) conjugate if there exists a $ C^k $ (resp. analytic, formal) coordinate change between them. A $ C^k $ (resp. analytic, formal) normal form is a choice of representative for each of the conjugacy classes. For this reason, normal form theory plays a crucial role in understanding the local behaviour of vector fields near a singularity or invariant manifold. A reasonably exhaustive account of the modern theory is given in \cite{murdock2006normal}. 
		
		The utility of normal forms has led many authors to develop several styles of normal forms; for instance \cite{brunoLocalMethodsNonlinear1989,elphickSimpleGlobalCharacterization1987,belitskii2002c}. The most common are the semi-simple and inner-product styles. The semi-simple style is advantageous when the Jacobian at the singularity is semi-simple, whilst the inner-product is useful when there is some nilpotent component or when the Jacobian vanishes. 
		
		There are no theoretical barriers to using the inner-product style, particularly the work of Stolovitch and Lombardi \cite{lombardi2010normal}, to study normal forms for singularities in a normally hyperbolic manifold. However, in Section \ref{sec:FormalNormalForms}, a new style of normal form will be derived which takes advantage of the centre subspace. The normal form is considered through an algebraic lens, akin to \cite{murdock2006normal}. The new approach provides results which are analogous to normal forms for hyperbolic singularities, namely, resonance conditions which describe the irremovable monomials in Lemma \ref{lem:resCondition}, and Theorem \ref{thm:FormalNormalForm} which categorises the formal normal form near normally hyperbolic invariant manifolds.
		
		Normal forms are then studied in the $ C^k $ category. Using a crucial theorem of Belitskii and Samavol \cite{ilyashenkoNonlocalBifurcations1998}, a proof is given of Corollary \ref{cor:CkNormalFOrm} on the existence of a $ C^k $ transformation bringing a vector field normally hyperbolic to a manifold of singularities into truncated normal form. In the smooth case, the result is analogous to the Sternberg-Chen Theorem for hyperbolic singularities \cite{sternbergStructureLocalHomeomorphisms1958,chenEquivalenceDecompositionVector1963}. The new style of normal form derived in Section \ref{sec:FormalNormalForms} is crucial to the proof. The result extends previous work by Takens \cite{takensPartiallyHyperbolicFixed1971} which covers the non-resonant case in a finite class of differentiability.
		
		With the normal form theory detailed, we then study Dulac maps near normally hyperbolic saddles in Section \ref{sec:TransMap}. The investigation is motivated by the many applications in \cite{duignanC83regularisationSimultaneous2020,roussarieAlmostPlanarHomoclinic1996,caillauSingularitiesMinTime2018}. Specifically, these works demand asymptotic properties of the transition map between sections transverse to the centre-stable and centre-unstable manifolds of the normally hyperbolic manifold. All applications require only a study of the case when either the stable or unstable manifold of each point on the normally hyperbolic manifold is of dimension 1. Thus we restrict our attention to this case.
		
		The Dulac map for families of hyperbolic saddles in the plane has been studied extensively. For an overview see \cite{roussarieBifurcationPlanarVector1998}. 	
		Dulac maps near a family of hyperbolic saddles in $ \R^3 $ have been treated in \cite{bonckaertAsymptoticPropertiesDulac2001,roussarieAlmostPlanarHomoclinic1996} and for some special saddle points in \cite{dumortierBifurcationsCuspidalLoops1997a}. In \cite{caillauSingularitiesMinTime2018} the Dulac map near a specific manifold of normally hyperbolic saddle singularities was studied. The asymptotic structure of the Dulac maps in the general case is heretofore not investigated. 
		
		In Section \ref{sec:TransMap} we prove Theorem \ref{thm:asymStructureofDNinN} and \ref{thm:asymptoticStructureOfDN} on the asymptotic structure of the transition map. It is shown that the transition map shares properties with the familiar planar case. In particular, the Dulac map has a Mourtada type structure \cite{mourtada1990cyclicite} and is an  asymptotic series in terms of the form,
		\begin{equation*}
		\omega(\alpha,x) = \begin{cases}
			\frac{x^{-\alpha} - 1}{\alpha}, & \alpha \neq 0 \\
			-\ln x	& \alpha =0
		\end{cases},
		\end{equation*}
		with $ x $ some small coordinate on the section and $ \alpha $ a parameter dependent on the eigenvalues of the Jacobian on the normally hyperbolic manifold.
		
	\section{Normal Forms}\label{sec:NormalForms}
		We first give some notations. Let $ \K $ be the field of real $ \R $ or complex $ \C $ numbers. Suppose $ x = (x_1,\dots,x_k) \in \K^k $ and denote by $ \partial_x := (\partial_{x_1},\dots,\partial_{x_k}) $. Then, given a function $ f:\K^k\to\K^k $, a vector field $ X $ on $ \K^k $ is defined by 
		\[ X = f\partial_x := f_1\partial_{x_1} + \dots + f_k \partial_{x_k}.  \]
		Furthermore, if $ \alpha = (\alpha_1,\dots,\alpha_k) \in \N^k $ the multinomial notation $ x^\alpha $ will be used to represent the monomial $ x_1^{\alpha_1}\dots x_k^{\alpha_k} $ of degree $ |\alpha| := \alpha_1 +\dots + \alpha_k  $.
		
		\subsection{Formal Normal Forms}\label{sec:FormalNormalForms}
			In this section the necessary theory to state and prove Theorem \ref{thm:FormalNormalForm} on formal normal forms for manifolds of normally hyperbolic singularities is built. Take $ X $ to be a germ of a smooth $ C^\infty $ or analytic $ C^\omega $ vector field on $ \K^n $ that is normally hyperbolic along an invariant manifold $ \NHIM $ of dimension $ k $ consisting entirely of singular points.
			
			A \textit{pre-normal form} can be constructed for $ \NHIM $ from well known results in the literature. In a neighbourhood of any point $ u_0 \in \NHIM $ there exists a $ C^\infty $ transformation straightening $ \NHIM $ and aligning the stable-centre $ W^{sc}(\NHIM) $ and unstable-centre $ W^{uc}(\NHIM) $ manifolds with coordinate axis \cite{wigginsNormallyHyperbolicInvariant1994}. That is, coordinates $  (x,u)\in\K^{n-k}\times\K^k $ local to $ u_0 = 0 $ can be taken such that $ X $ is of the form,
			\begin{equation}\label{eqn:preNormalForm}
				X = \left(A(u) x + f(x,u)\right) \partial_x + g(x,u) \partial_u,\quad f(0,u) = g(0,u) = 0.
			\end{equation}
			Note that in this pre-normal form $ \NHIM = \{x = 0\} $ and hence $ u $ are the centre variables. Using the theory in \cite{wigginsNormallyHyperbolicInvariant1994} further geometric properties on $ f,g $ and $ A $ can be assumed, however, for the purposes of this paper they do not play a central role. In what follows, assume that $ X $ is in this pre-normal form.
			
			In standard normal form theory one would now proceed by introducing the formal Taylor series of $ X $ at $ 0 $ in $ (x,u) $ and analyse which terms can be removed by a formal, near identity coordinate transformation $ \hat{\phi} $. Much theory has been developed in this avenue. Although these methods can certainly be implemented here, particularly the work of \cite{belitskii2002c,lombardi2010normal}, the degeneracy of the flow on $ \NHIM $ enables a slight modification of the methods and leads to a normal form with more removable terms than the standard theory.
			
			The key modification is to take a series expansion only in the normal variables $ x $ instead of all the variables $ (x,u) $. This produces a series expansion about $ x=0 $ of the form,
			\begin{equation}\label{eqn:seriesexpansion}
				X \sim X_0(u;x) + X_1(u;x) + \dots, \qquad X_0(u;x) = A(u) x \partial_x + 0\cdot \partial_u,
			\end{equation}
			where each $ X_d(u;x) $ is of dimension $ n $ and each component is a degree $ d+1 $ homogeneous polynomial in $ x = (x_1,\dots,x_{n-k}) $ with coefficients that are functions in $ u $. These coefficient functions can be considered either formal, smooth, or analytic in a neighbourhood of $ u=0 $ if $ X $ is respectively formal, smooth, or analytic.
			
			With some notation identified, the algebraic structure of the series expansion \eqref{eqn:seriesexpansion} can be formulated.
			\begin{definition}
				Define the following algebraic objects:
				\begin{enumerate}[i.]
					\item $ \hat{C}^\infty(u)$ the ring of formal power series of $u\in\K^k$. $ C^\infty(u), C^\omega(u) $ the ring of germs of respectively smooth, analytic functions in a neighbourhood of $ 0\in \K^k $. Denote all three by $ C $.
					\item $ C\mathcal{P}_d $ the free $ C $-module generated by the set of degree $ d+1 $ monomials in $ x $.
					\item $ C\H_d $ the free $ C $-module given by $ n $ copies of $ C\mathcal{P}_d $. Consider each element of $ C\H_d $ as an $ n $-dimensional vector space with components homogeneous polynomials of degree $ d+1 $ in $ x $ and whose coefficients are $ C $ functions in $ u $.
					\item $ C\H $ the Lie algebra of $ n $ dimensional formal vector fields in $ x $ with coefficients in $ C $. We take the usual Lie bracket $ [\cdot,\cdot] $ for vector fields.
					\item $ C\mathcal{F} $ the associated Lie group of $ C\H $.
				\end{enumerate}
			\end{definition}
		
			With these definitions, \eqref{eqn:seriesexpansion} can now be seen as identifying $ X $ with a formal germ of a vector field $ \hat{X}\in C\H $ and decomposing $ \hat{X} $ into $ X_d(u;x) \in C\H_d $. In what follows, germs of vector fields $ \hat{X}\in C\H $ are considered in order to produce a result on formal normal forms. This provides a succinct Lie algebraic approach to the theory. In Section \ref{sec:CkNormalForms}, properties about the actual germ $ X $ are recovered. 
			
			As detailed in \cite{murdock2006normal}, formal, near identity transformations $ \hat{\phi} \in C\mathcal{F} $ can be constructed via a generating vector field $ U \in C\H $ by taking $ \hat{\phi} $ the time $ 1 $ flow of $ U $. Moreover, one can pull back $ \hat{X} \in C\H $ to produce the transformed vector field $ \tilde{X} $ through the relation,
			\begin{equation}\label{eqn:conjugation}
				\tilde{X} = \exp(L_U )\hat{X}, \qquad L_U := [U,\cdot].
			\end{equation}
			
			Note that $ \hat{\phi} $ is in general a divergent series in $ x $ and thus only a formal transformation. However, one can write the expansion so that the coefficients of the $ x $ terms are functions in $ C(u) $. Using $ \exp(L_U) $ is particularly useful to preserve a Hamiltonian structure, see for instance \cite{Siegel2012}, but it is being used here in the general sense.
			
			In line with the usual normal form theory,  a \textit{cohomological equation} on each $ C\H_d $ will now be constructed from \eqref{eqn:conjugation}. A consequent examination of the cohomological equations will reveal which monomial vector terms in $ \hat{X} $ can be removed by a formal transformation $ \hat{\phi} $. 
			
			Let $ U_d \in C\H_d $ and transform $ \hat{X} $ by the generated transformation $ \hat{\phi}_d $ to obtain,
			\begin{align*}
				\tilde{X} 	&= \exp(L_{U_d}) \hat{X} \\
							&= (Id + L_{U_d} + \dots) (X_0 + X_1 + \dots + X_d + \dots) \\
							&= (X_0 + X_1 + \dots + X_d + [U_d,X_0] + \dots ).
			\end{align*}	
			The first terms influenced by the transformation $ \hat{\phi}_d $ is at order $ d $ and produces the equation
			\begin{equation}
				[X_0,U_d] = X_d - \tilde{X}_d.
			\end{equation}
			
			However, if $ U_d \in C\H_d $ it is not necessarily true that so too is $ [X_0,U_d] $. To see this, let a vector field $ X $ act on a vector field $ U $ by treating $ X $ as a derivation on each coordinate function and let $ U = U^x \partial_x + U^u \partial_u $. Then,
			\begin{align*}
				[X_0,U_d] 	&=  X_0(U_d) - U_d(X_0) \\
							&= \left(A(u)x \partial_x \right)(U_d^x \partial_x + U_d^u \partial_u) - (U_d^x \partial_x + U_d^u \partial_u)\left(A(u)x \partial_x \right) \\
							&= \left( A(u)x \partial_x (U_d^x) - U_d^x\partial_x(A(u)x) \right)\partial_x + \left( A(u)x\partial_x U_d^u \right)\partial_u - U_d^u\partial_u(A(u))x) \partial_x.
			\end{align*}
			The terms $$ \tilde{L}_d (U_d^x \partial_x) := [X_0, U_d^x\partial_x] = \left( A(u)x \partial_x (U_d^x) - U_d^x\partial_x(A(u)x) \right)\partial_x $$ and $$ X_0(U_d^u\partial_u) = \left( A(u)x\partial_x U_d^u \right)\partial_u $$ are both in $ C\H_d $. The final term $$ U_d^u\partial_u(A(u)x) \partial_x  $$ is in  $ C\H_{d+1} $. If this final term is pushed into the higher order terms of the expansion, then the effect of $ U_d $ on $ \hat{X} $ has first influence at degree $ d $ and is quantified by the \textit{modified cohomological equation}
			\begin{equation}
				\hat{L}_d (U_d) = X_d - \tilde{X}_d,
			\end{equation}
			with 
			\begin{equation*}
				\hat{L}_d := \tilde{L}_d \oplus X_0, \qquad \tilde{L}_d \in \End(C\H_d^x), \qquad	X_0 \in \End(C\H_d^u)
			\end{equation*}
			and $ C\H_d^x,\ C\H_d^u  $ are the submodules with vanishing $ u $ and $ x $ components respectively. 
			
			\begin{remark}
				It is worth pointing out the difference between the \emph{modified} cohomological equation and the usual cohomological equation in the normal form theory using the semi-simple or inner-product styles. The usual cohomological equation is of the form,
				\[ L_d (U_d) = X_d - \tilde{X}_d, \]
				with $ L_d := [X_0,\cdot] $. In the usual styles one has each $ X_d \in \H_d $, the vector space of degree $ d+1 $ homogeneous vector fields. With this grading $ L_d:\H_d\to\H_d $. The fact that $ L_d $ is an endomorphism on $ \H_d $ is crucial to constructing an iterative scheme on the degree $ d $, which in turn construct the normal form. However, in the new approach of this paper, we have decomposed the vector field $ X $ through the grading $ X_d\in C\H_d $, the $ C $-module of germs vector fields homogeneous in $ x $ only. In the above calculation, it is shown that $ L_d(U_d) $ produces a term $ U_d^u \partial_u(A(u)x)\partial_x \in C\H_{d+1} $. Thus, $ L_d $ acting on $ C\H_d $ is not an endomorphism. Ignoring the higher order term $ U_d^u \partial_u(A(u)x)\partial_x $ produces the endomorphism $ \hat{L}_d $ as desired.
			\end{remark}
			
			\begin{remark}\label{rmk:operatorsAreMatrices}
				A choice of ordering of the degree $ d+1 $ monomials vectors $ x^\alpha  := x^{\alpha_1}\dots x^{\alpha_{n-k}}, |\alpha| := \alpha_1 + \dots + \alpha_{n-k} = d+1 $ creates a basis for $ C\mathcal{P}_d $. Then, by ordering each vector component $ \partial_{x_i},\partial_{u_i} $ together with the ordering of $ C\mathcal{P}_d $, a basis for $ C\H_d $ can be obtained. Let the dimension of $ C\H_d $ be $ D(d) $. As $ C\H_d $ is a free module over $ C $, we have $ C\H_d \cong (C)^{D(d)} $. Thus, with a choice of basis, one can consider $ \hat{L}_d $ as a $ D(d) $ square matrix with entries in $ C $, that is, $ \End(C\H_d) \cong M_{D(d)}(C) $.	
			\end{remark}
			
			With the modified cohomological equation derived, terms in $ X_d $ removable by some formal transformation $ \hat{\phi}\in C\mathcal{F} $ can now be determined. In fact, it should be evident that all terms of $ X_d $ that are in $ \Ima(\hat{L}_d) $ can be removed by a choice of $ U_d $, and conversely, any component of $ X_d $ in $ C\H_d \setminus \Ima(\hat{L}_d) $ are irremovable. By taking $ \tilde{X}_d $ equal to the sum of these irremovable terms, it can be assured that $ X_d - \tilde{X}_d \in \Ima(\hat{L}_d) $ and the modified cohomological equation at order $ d $ can be solved. Formally, one takes the quotient module
			\[ \coker(\hat{L}_d) := \bigslant{C\H_d }{\Ima(\hat{L}_d)} \]
			and a choice of representatives $ \tilde{X}_d $ of elements $ [\tilde{X}_d] \in \coker(\hat{L}_d) $. In the terminology introduced by Murdock \cite{murdock2006normal}, this choice of representative is considered a \textit{normal form style}.
			
			In summary, it has been shown that a formal normal form for $ \hat{X} $ can be constructed through an iterative procedure. Assuming $ \hat{X} $ has been normalized to order $ d-1 $, generate a formal, near identity transformation $ \phi_d $ from a vector field $ U_d \in C\H_d $. The pull-back of $ \hat{X} $ by $ \phi_d $ leaves terms of order $ d-1 $ unchanged and produces at order $ d $ the modified cohomological equation. Then, one removes all terms from $ X_d $ that are contained in $ \Ima(\hat{L}_d) $ and the normalized terms become a choice of representative from $ \coker(\hat{L}_d) $. The procedure is repeated for $ d+1 $. The following central theorem has thus been proved.
			
			\begin{thm}\label{thm:generalFormalNormalForm}
				Let $ X $ be a germ of a $ C^\infty $ vector field that is normally hyperbolic on a manifold of singularities $ \NHIM $ and let $ \hat{X} $ be the corresponding formal series of $ X $ at 0.
				Then there exists a sequence of transformations $ \phi_d $ generated by $ U_d\in C\H_d $ which formally conjugates $ \hat{X} $ to the normal form,
				\begin{equation}
					\tilde{X} =  X_0 + \sum\limits_{d\geq 1} \tilde{X}_d,
				\end{equation}
				with $ \tilde{X}_d $ a representative of $ [\tilde{X}_d] \in \coker(\hat{L}_d) $.
			\end{thm}
			
			Whilst Theorem \ref{thm:generalFormalNormalForm} gives the algebraic structure of the normal form for a vector field $ X $, it does little to give a more concrete explanation of what terms $ \tilde{X}_d $ look like or how to find and choose the precise representative. Crucially, we want to know in what situations it can be assumed that $ \tilde{X}_d = 0 $, that is, we want to know a simple way of determining when $ X_d \in \Ima(\hat{L}_d) $. 
			
			Answers are provided in the case $ A(u) $ is diagonalisable. In this case it may be assumed that $ A(u) = \diag\left( \lambda_1(u),\dots,\lambda_{n-k}(u) \right) $ and by hyperbolicity each $ \re \lambda_i(0) \neq 0 $. Lemma \ref{lem:diagonalOperator} follows.
			
			\begin{lemma}\label{lem:diagonalOperator}
				Suppose $ X_0 = A(u)\partial_x $ and $ A(u) = \diag(\lambda_1(u),\dots,\lambda_{n-k}(u)) $. Then each modified homological operator $ \hat{L}_d \in \End(C\H_d) $ is diagonal. More precisely, if $ \alpha \in \N^{n-k},\ |\alpha| = d+1 $, $ \lambda(u) := (\lambda_1(u),\dots,\lambda_{n-k}(u)) $, and $ \langle \cdot,\cdot \rangle $ is the usual dot product on $ \K^{n-k} $, then
				\begin{equation}\label{eqn:resonanceEqn}
					\begin{aligned}
						\hat{L}_d(x^\alpha \partial_{x_i})	&= \left (\left\langle \lambda(u), \alpha \right\rangle - \lambda_i(u)\right) x^\alpha \partial_{x_i}, \\
						\hat{L}_d(x^\alpha \partial_{u_i}) 	&=  \left\langle \lambda(u), \alpha \right\rangle  x^\alpha \partial_{u_i}.
					\end{aligned}
				\end{equation} 
			\end{lemma}
			\begin{proof}
				This is a calculation using the definition of $ \hat{L}_d $.
			\end{proof}
			
			Let $ v $ denote $ x_i $ or $ u_i $. Then $ C\H_d $ admits submodules $ C\H_{\alpha,v} $, each defined as the free module over $ x^\alpha \partial_v $ and all of which are isomorphic to $ C $. Hence, Lemma \ref{lem:diagonalOperator} reduces the problem of describing $ \Ima(\hat{L}_d) $ into a study of the endomorphisms $ L_{\alpha,v} \in \End(C\H_{\alpha,v}) \cong \End(C)  $ and their images. These endomorphisms act by mere multiplication of $ f_{\alpha,v}(u) $ on $ C $, where $ f_{\alpha,v}(u) $ is given by the coefficient of $ x^\alpha \partial_v $ in \eqref{eqn:resonanceEqn}. Finding a representative of $ \coker(\hat{L}_d) $ is reduced to finding representatives of \[ \coker(L_{\alpha,v}) = \bigslant{C\H_{\alpha,v}}{\Ima(L_{\alpha,v})}. \]
			
			The image $ \Ima(L_{\alpha,v}) $ is equivalent to the ideal generated by $ f_{\alpha,v} $, namely $ \langle f_{\alpha,v} \rangle $. It follows, if $ f_{\alpha,v} $ has a multiplicative inverse, that is, $ f_{\alpha,v} $ is a unit, then $ \Ima(L_{\alpha,v}) = C\H_{\alpha,v} $. Consequently, $ \coker(L_{\alpha,v}) = 0 $ and the unique representative $ 0 $ can be chosen. The following lemma is analogous to the usual resonance conditions for normal forms of hyperbolic singular points.
			
			\begin{lemma}\label{lem:resCondition}
				Suppose $ A(u) = \diag(\lambda_1(u),\dots,\lambda_{n-k}(u)) $. Then all terms of the form,
				\begin{equation}
					\begin{aligned}
						f(u) x^{\alpha} \partial_{x_i}, &\quad \langle \alpha,\lambda(0) \rangle - \lambda_i(0) \neq 0 \\
						f(u) x^{\alpha} \partial_{x_i}, &\quad \langle \alpha,\lambda(0) \rangle \neq 0
					\end{aligned}
				\end{equation} 
				do not appear in the normal form $ \tilde{X} $.
			\end{lemma}
			\begin{proof}
				From Theorem \ref{thm:generalFormalNormalForm} a normal form transformation can be found which brings the coefficient of $ x^\alpha \partial_v $ to a representative of $ [f(u)] \in \coker(L_{\alpha,v}) $. If it can be shown that $ f_{\alpha,v} $ is a unit then the remarks of the proceeding exposition show this representative can be taken as $ 0 $. The units of $ C $ are easily described as the functions $ g(u) $ such that $ g(0) \neq 0 $. Now, $ f_{\alpha,v}(u) = \langle \alpha,\lambda(u) \rangle - \lambda_i(u)  $ when $ v = x_i $ and $ \langle \alpha,\lambda(u) \rangle $ when $ v = z_i $, thus the lemma can be concluded.
 			\end{proof}
		
			\begin{definition}
				The vector monomials in the union of the sets,
				\begin{equation}
					\begin{aligned}
						\Res_x &:= \{x^\alpha\partial_{x_i} \ |\ \langle \alpha,\lambda(0) \rangle - \lambda_i(0) = 0 \}, \\
						\Res_u &:= \{x^\alpha\partial_{u_i} \ |\ \langle \alpha,\lambda(0) \rangle = 0 \}, \\
						\Res_d &:= \{ x^\alpha \partial_v\in \Res_x\cup\Res_u\ |\ |\alpha| = d+1 \}
					\end{aligned}
				\end{equation} 
				are called \textit{resonant}. Moreover, the free $ C $-submodule over the set $ \Res_d $ is denoted by $ C\Res_d $ and called the \textit{resonant submodule of order $ d $}.
			\end{definition}
			
			The final problem to be resolved concerns these resonant terms. They can not a priori be removed and a choice of representative must be made. A concrete explanation of the problem of choosing a representative is, given a function $ F(u) \in C $, finding $ q(u), r(u) \in C $ such that \[ F(u) = r(u) + q(u) f_{\alpha,v}(u). \] 
			In the normal form procedure, $ F(u) $ is the coefficient of $ x^\alpha\partial_v $ in $ X_d $ and choosing an $ r(u) $ amounts to choosing a representative of $ [F(u)] \in \coker(L_{\alpha,v}) $. The question is now, is it possible to do this quotient? Of course, one can always take $ r(u) = F(u) $ and $ q(u) = 0 $, but this may not be the `simplest' form of $ r(u) $. For instance, if $ F(u) = f_{\alpha,v}(u) $, clearly a better choice is $ r(u) = 0,\ q(u) = 1 $. The following divisibility theorem provides what may be called the simplest form of $ r(u) $.
						
			\begin{thm}[Weierstrass/Mather Division Theorem \cite{golubitskyStableMappingsTheir1973}]\label{thm:division}
				Let $ f $ be a smooth (resp. analytic or formal) $ \K $-valued function defined on a neighbourhood of $ 0 $ in $ \K\times\K^{k-1} $ such that $ f(u_1,0) = u_1^m g(u_1) $ where $ g(0) \neq 0 $ and $ g $ is smooth (resp. analytic or formal) on some neighbourhood of $ 0 $ in $ \K $. Then given any smooth (resp. analytic or formal) real-valued function $ F $ defined on a neighbourhood of $ 0 $ in $ \K\times\K^{k-1} $, there exist smooth (resp. analytic or formal) functions $ q $ and $ r $ such that
				\begin{enumerate}[(i)]
					\item $ F = r + q f $ on a neighbourhood of $ 0 $ in $ \K\times\K^{k-1} $, and
					\item $ r(u) = \sum_{i=0}^{m-1} r_i(u_2,\dots,u_k) u_1^i $. 
				\end{enumerate}
			\end{thm}
			\begin{remark}
				When $ f \neq 0 $ is a formal or analytic function on $ \K^k $ then, possibly after a linear change of $ u $, there is always an $ m $ and a $ u_i $ such that $ f(u_i,0) = u_i^m g(u_i) $. The value of $ m $ is given by the first non-zero $ m $-jet of $ f $. Moreover, it is shown in \cite{golubitskyStableMappingsTheir1973} that $ q,r  $ are unique. Algebraically, this means a unique representative of each element in $ \coker(\hat{L}_d) $ can be taken for $ C = \hat{C}^\infty $ or $ C^\omega $.
			\end{remark}
			\begin{remark}
				Uniqueness of the functions $ r,q $ fails when $ f $ is $ C^\infty $. The issue is the existence of $ f \neq 0 $ such that the $ \infty $-jet is $ 0 $, so called \textit{flat functions}. A counterexample is given in \cite{golubitskyStableMappingsTheir1973}. Take $ f $ polynomial, $ F = 0 $, and $ G $ flat. Then both $ r_1 = 0 = q_1 $ and $ r_2 = G, q_2 = -G/f $ satisfy $ F = r + q f $ and are smooth. Algebraically, this means a unique representative of each element in $ \coker(\hat{L}_d) $ when $ \hat{L}\in \End(C^\infty\H_d) $ can not be be given by Theorem \ref{thm:division}. However, a choice of representative can be made by decomposing $ F = \hat{F} + \bar{F}, f= \hat{f} + \bar{f} $ where $ \hat{\cdot},\ \bar{\cdot} $ represent the formal and flat part respectively. $ r $ can be chosen as the unique formal function given by Theorem \ref{thm:division} and satisfying $ \hat{F} = \hat{r} + q \hat{f} $. The flat terms can then be added to get an $ r = \hat{r} + \bar{r}, \bar{r} = \bar{F} - q \bar{f} $. For the counterexample, this forces the choice of $ r = q = 0 $.
			\end{remark}
			
			The main theorem for diagonalisable $ A(u) $ has thus been proved.
			\begin{thm}\label{thm:FormalNormalForm}
				Let $ X $ be a germ of a vector field of class $ C = \hat{C}^\infty, C^\infty,$ or $ C^\omega $ that is normally hyperbolic on a manifold of singularities $ \NHIM $, and let $ \hat{X} \in C \H $ be the corresponding formal series of $ X $.
				Then there exists a sequence of transformations $ \phi_d $ generated by $ U_d\in C\H_d $ which formally conjugates $ \hat{X} $ to the normal form,                          
				\begin{equation}
					\tilde{X} =  X_0 + \sum\limits_{d\geq 1} \tilde{X}_d,
				\end{equation}
				with $ \tilde{X}_d \in C\Res_d $ whose coefficients are of the form $ r(u) $ given in Theorem \ref{thm:division}. In particular, if $ X $ is analytic or formal then $ r(u) $ is polynomial in at least one of the $ u_i $.
			\end{thm}
		
		\subsection{$ C^k $-Normal Forms}\label{sec:CkNormalForms}
			Theorem \ref{thm:FormalNormalForm} provides a formal normal form $ \tilde{X} $ for a given germ of a vector field $ X $ near a point $ u_0 $ of a normally hyperbolic manifold of singularities $ \NHIM $. The theorem states the existence of a formal transformation $ \hat{\phi} $ bringing $ \hat{X} $ into its normal form $ \tilde{X} $. However, the statement is only formal, meaning that $ \tilde{X} \sim \hat{\phi}^* X $ where $ \sim $ is equivalence of the series expansion at $ 0 $ in one of the forms \eqref{eqn:seriesexpansion}. There are two questions worth addressing:
			\begin{enumerate}
				\item Can $ \hat{\phi} $ be taken smooth or analytic?
				\item If $ \tilde{X}^K := X_0 + \sum_{d \leq K } \tilde{X}_d $ is the normal form of $ X $ truncated at degree $ K $, does there exist an integer $ k $ and $ \phi\in C^k $ which conjugates $ X $ to $ \tilde{X}^K $?
			\end{enumerate} 
		
			The usual trick to replace a formal transformation $ \hat{\phi} $ with a smooth transformation $ \phi $ is to evoke the Borel extension lemma \cite[pg.~98, Lemma~2.5]{golubitskyStableMappingsTheir1973}. The lemma guarantees, for any formal series $ \hat{\phi} $, the existence of a smooth function $ \phi \sim \hat{\phi} $. If this lemma can be applied here, then there is a smooth transformation $ \phi $ such that $ \tilde{X} \sim \phi^* X $. 
			
			In order to apply the Borel lemma to a transformation $ \hat{\phi}\in C \mathcal{F} $, each of the coefficient functions from $ C $ must be defined on the same domain. In general, this is impossible! The problem comes from the possibility of other resonances occurring when the spectrum $ \lambda(u) \in \C^k $ of $ A(u) $ depends on $ u $. That is, a resonance of the form
			\[ \langle \alpha, \lambda(u)\rangle - \lambda_i(u) = 0 \text{ or } \langle \alpha, \lambda(u)\rangle, \]
			for $ u \neq 0 $ for some $ \alpha \in \N^k $. Such an additional resonance will shrink the domain on which $ f_{a,v}(u) $ is an identity, and hence the domain for which the coefficients of $ \hat{\phi} $ are smooth.
			
			Nevertheless, the following lemma can be proved		
			\begin{lemma}\label{lem:neighbourhoods}
				Let $ X $ be a germ of a $ C^\infty $ vector field that is normally hyperbolic on a manifold of singularities $ \NHIM $. Then there exists a sequence of neighbourhoods $ W_d $ of $ (0,0) $ with $ W_{d+1} \subset W_{d} $, and a sequence of transformation $ \phi^d $, polynomial in $ x $ and smooth in $ W_d $, such that, for any $ K \in \N $, 
				\[ \phi^{K*} X = X_0 + \sum_{d=1}^{K}\tilde{X}_d + R_K, \]
				where $ \tilde{X}_d \in C\Res_d $ and $ R_K $ is $ K $-flat in $ x $.
				
				Moreover, if $ \cap_{d\geq 1} W_d $ contains some open neighbourhood $W_\infty$ of $\NHIM $, then there exists a function $ \phi $ smooth in a neighbourhood $ W_\infty $ of $ 0 $ so that,
				$ \phi^{K*} X = X_0 + \sum_{d\geq 1}\tilde{X}_d + R_\infty, $
				where $ R_\infty $ is flat in $ x $.
			\end{lemma}
			\begin{proof}
				For each $ K < \infty $, we can always take a sufficiently small neighbourhood $ W_K $ of $ (x,u) = 0 $ so that there are no resonance conditions for $ (x,u)\in W_k $ with $ u \neq 0 $. Hence, the coefficients of $ \hat{\phi} $ for each monomial of degree less than $ K $ are smooth in $ W_k $. By truncating $ \hat{\phi} $ at order $ K $, from Theorem \ref{thm:generalFormalNormalForm} we obtain a polynomial transformation $ \phi^K $ which is smooth in the neighbourhood $ W_K $ and with the desired conjugation properties.
				
				If $ \cap_{d\geq 1} W_d  $ contains some open neighbourhood $W_\infty$ then this is a common domain for which all coefficient functions of $ \hat{\phi} $ are smooth. The Borel extension lemma concludes the result.
			\end{proof}	
		
			\begin{remark}
				There are some important cases which guarantee the application of the
        Borel lemma. For example, if the spectrum $ \lambda(u) $ is constant in
        a neighbourhood of $ 0 $, or of the form $ \kappa(u) \lambda $ for some
        smooth scalar function $ \kappa $, or if the eigenvalues are purely
        attracting (or repelling).
			\end{remark}	
			
			
			The question remains, if $ \phi $ can only be assumed smooth or polynomial in general, whether the remainder term $ R_K $ can be removed so that formal conjugacy can be replaced by smooth conjugacy. In the case of a purely hyperbolic singularity, the question is answered positively by the Sternberg-Chen Theorem \cite{sternbergStructureLocalHomeomorphisms1958,chenEquivalenceDecompositionVector1963}. 
			
			A more general problem is, given two vector fields $ X,\tilde{X} $ with identical $ K(k) $-jet at $ 0 $, when can it be guaranteed $ X,\tilde{X} $ are $ C^k $ conjugate for some function $ K:\N\to\N $. The most general theorem in this direction has been proved for maps by Samovol and for vector fields by Belitskii.
			
			\begin{thm}[Belitskii-Samovol \cite{ilyashenkoNonlocalBifurcations1998}]\label{thm:Belitskii-Samovol}
				For any $ k\in\N $ and any tuple $ \lambda \in \C^n $ there exists an integer $ K = K(k,\lambda) $ such that the following holds. Suppose two germs of vector fields at a singularity with the spectrum of	linearization equal to $ \lambda $ have a common centre manifold, and their jets of order $ K $
				coincide at all the points of this manifold. Then these germs are $ C^k $ equivalent.
			\end{thm}
		
			Hartman, in \cite{hartmanOrdinaryDifferentialEquations2002}, proved a version of this theorem with $K(k, \lambda)$ explicitly given as an affine
			function of $k$ and with coefficients in terms of $\lambda$. In the original proof by Belitskii, there is also an explicit expression of $K(k,\lambda)$ which is optimal and depends on the gaps between the real parts of the eigenvalues. The less explicit version stated here is proved in \cite{ilyashenkoNonlocalBifurcations1998} and uses the `path' or `homotopy method'. This method of proof allows one to take $ k = \infty $ provided one first has only a flat remainder as in Lemma \ref{lem:neighbourhoods}. A similar proof to that in \cite{ilyashenkoNonlocalBifurcations1998} which explicitly gives the $ k=\infty $ case was given in \cite[Thm.~10]{roussarie1975modeles} for families of hyperbolic singularities. 
		
			Theorem \ref{thm:Belitskii-Samovol} can be applied provided the $ K(k) $-jets of $ X $ and $ \tilde{X} $ agree along $ x=0 $ in a neighbourhood of $ (x,u)= 0 $. Indeed this is true for any $ \phi^{K*}X $ as the remainder $ R_K $ is $ K $-flat along $ x $. Hence, the following key corollary on the $ C^k $-normal form near points in $ \NHIM $ has been shown.
		
			\begin{corollary}\label{cor:CkNormalFOrm}
				Let $ X $ contain a manifold of normally hyperbolic singularities $ \NHIM $. Then there exists a function $ K(k) : \N \to \N $ such that $ K(k) \to \infty $ as $ k\to \infty $, and such that $ X $ is $ C^k $-conjugate to the normal form $ X^{K(k)} $ in a neighbourhood $ W_K $ of any point $ p \in \NHIM $.
				
				Moreover, one can take $ K = \infty $ if, in a neighbourhood of $ p $, the spectrum $ \lambda(u) $ of $ A(u) $ is constant, or of the form $ \kappa(u) \lambda $ for some smooth scalar function $ \kappa $, or if the eigenvalues are purely attracting (or repelling).
			\end{corollary}
		
			Finally, we give comment to the case $ X $ is analytic. If $ \phi $ can be
      taken analytic then both proposed questions are answered. A substantial
      amount of work in the literature has already addressed the potential
      analyticity of $ \phi $ for a hyperbolic singularity, for an overview see
      \cite{walcherSymmetriesConvergenceNormal2004}. In this context, provided
      the eigenvalues of the Jacobian at the singularity satisfy the Bruno
      conditions, analyticity is guaranteed. The condition also holds for
      families of vector fields. Analyticity is not of concern in this paper,
      but due to the similarity in the resonance conditions between normal forms
      for hyperbolic singularities and normal forms for normally hyperbolic sets
      of singularities, we conjecture an analogous condition holds. This
      conjecture is further evidenced by the recent result in
      \cite{duignanChazyTypeAsymptoticsHyperbolic2020} which contains a theorem
      guaranteeing analyticity of the normal form in the case that $\lambda(u) =
      \kappa(u)\tilde{\lambda},\, \tilde{\lambda}\in\C,\, \kappa(u)\in\C^k $.
		
	\section{Asymptotic Properties of the Transition Map Near Some Normally Hyperbolic Saddles}\label{sec:TransMap}
		
	  In this section we derive the asymptotic properties of transitions near a
	  manifold $ \NHIM $ of normally hyperbolic singularities and provide a method
	  to compute them. We assume that at each point $ u_0\in\NHIM $ the eigenvalues
	  are real and there is at least one pair of eigenvalues of opposite sign, that
	  is, $ \NHIM $ contains normally hyperbolic saddles. Ideally, asymptotic
	  properties would be shown for arbitrary dimensions of the centre-stable $
	  W^{sc}(\NHIM) $ and centre-unstable $ W^{uc}(\NHIM) $ manifolds. However, a
	  derivation is given only when the unstable or stable manifold at each point $
	  u_0 \in \NHIM $ is one dimensional. Moreover, for clarity, focus is given only
	  on manifolds $ \NHIM $ of co-dimension 3. All methods introduced naturally
	  extend to the higher co-dimension cases. Remarks are given throughout for the
	  case $ \NHIM $ is co-dimension 2.
		
		Let $ X $ be a germ of a smooth vector field in a neighbourhood of a co-dimension 3 manifold $ \NHIM $ of normally hyperbolic saddle singularities. Let the dimension of $ \NHIM $ be $ k $. Without loss of generality assume that $ X $ is in the pre-normal form \eqref{eqn:preNormalForm} with $ (x,y,z)\in\R^3 $ so that $ \NHIM $ is given by $ (x,y,z)=0 $ and the centre variables are given by $ u \in \R^k $. By a time rescaling, it can be assumed that for all $ u \in \NHIM $ the eigenvalues of $ DX_u $ restricted to the normal space of $ \NHIM $ are given by $ (1, -\alpha(u), -\beta(u)) $ and satisfy,
		\[ -\alpha(0) \leq -\beta(0) < 0. \] Choose coordinates $ x,y,z $ so that
    the linearisation of the normal space is given by $ x \partial_x -\alpha(u)
    \partial_y -\beta(u) z \partial_z $. Note that if $ -\alpha(u) = -\beta(u) $
    then $ DX_u(0) $ may have some nilpotent component preventing this
    diagonalisation. This case is dealt with in the proceeding theory simply by
    treating the additional $ z \partial_y $ term as a higher order term.
		
		Before discussing the transitions of interest in this paper, it is useful to first classify the form of germs $ X $ in a neighborhood of a point on $ \NHIM $. This was accomplished in the previous section through normal form theory. The following proposition is an application of this work.
		\begin{proposition}\label{prop:normalForm3DimN}
			Let $ X $ be  the germ of a smooth vector field that is normally hyperbolic on a manifold of saddle singularities $ \NHIM $ as described above. For every point $ u_0\in\NHIM $ there exists a function $ K(k):\N\to\N $ such that, in some neighbourhood $ W_K $ of $ u_0 $, $ X $ is $ C^k $ conjugate to either \eqref{eqn:3DimNFninNu} or \eqref{eqn:3DimNFNu} subject to the following conditions.
			\begin{enumerate}[i)]
				\item Suppose that, $ \alpha(0) = \frac{p_1}{q_1} \in \Q,\ \beta(0) = \frac{p_2}{q_2} \in \Q $, $ \frac{\alpha(0)}{\beta(0)} \nin \N $ with both $ p_1,q_1 $ and $ p_2,q_2 $ co-prime. Let 
				\[ U_y = x^{\frac{p_1}{q_1}} y,\qquad U_z = x^{\frac{p_2}{q_2}} z. \]
				Under these resonance conditions $ X $ is conjugate to
				\begin{equation}\label{eqn:3DimNFninNu}
					\begin{aligned}
						\dot{x}	&= x \\
						\dot{y}	&= -\alpha(u) y +  y \sum_{1 \leq n_1 + n_2 \leq K} \alpha_{n_1,n_2}(u) U_y^{q_1 n_1} U_z^{q_2 n_2} \\
						\dot{z}	&= -\beta(u) z + z \sum_{1 \leq n_1+ n_2 \leq K} \beta_{n_1,n_2}(u) U_y^{q_1 n_1} U_z^{q_2 n_2}, \\
						\dot{u}_i &= \sum_{1 \leq n_1 + n_2 \leq K} \delta^i_{n_1,n_2}(u) U_y^{q_1 n_1} U_z^{q_2 n_2}, \quad i=1,\dots,k,
					\end{aligned}
				\end{equation}
				with $ n_1,n_2 \in \N $ and all functions in $u$ smooth. If $ \alpha(0) \nin \Q $ (resp. $ \beta(0) \nin \Q $) then there is no $ U_y $ (resp. $ U_z $) dependency.
				\item If additionally $ \frac{\alpha(0)}{\beta(0)} \in \N $ then there exists $ m,p,q \in \N $ with $ p,q $ co-prime such that $ \alpha(0) = m \frac{p}{q},\ \beta(0) = \frac{p}{q} $. Let
				\[ U_y = x^{\frac{mp}{q}} y,\qquad U_z = z^{\frac{p}{q}} y. \]
				Under these resonance conditions $ X $ is conjugate to
				\begin{equation}\label{eqn:3DimNFNu}
					\begin{aligned}
						\dot{x}	&= x \\
						\dot{y}	&= - \alpha(u) y + y \sum_{\substack{-1 \leq n_1 \leq K\\ q 0 \leq n_2 - m n_1 \leq K}} \alpha_{n_1,n_2}(u) U_y^{n_1} U_z^{q n_2 - m n_1}  \\
						\dot{z}	&= - \beta(u) z + z \sum_{\substack{ 0 \leq n_1 \leq K\\ -1 \leq q n_2 - m n_1 \leq K}} \beta_{n_1,n_2}(u) U_y^{n_1} U_z^{q n_2 - m n_1} \\
						\dot{u}_i &= \sum_{\substack{0 \leq n_1 \leq K\\ 0 \leq q n_2 - m n_1 \leq K}} \delta^i_{n_1,n_2}(u) U_y^{n_1} U_z^{q n_2 - m n_1}, \quad i=1,\dots,k,
					\end{aligned}
				\end{equation}
				with $ n_1,n_2 \in \N $ and all functions in $u$ smooth. If $ \alpha(0),\beta(0) \nin \Q $ then there is no $ U_y, U_z $ dependency. 
			\end{enumerate}
		\end{proposition}
		\begin{proof}
			As stated, the proposition is a direct consequence of Theorem \ref{thm:FormalNormalForm} on the normal form near a point in $ \NHIM $. It has been assumed that $ A(z) $ is diagonalised so that $ A(u) = \diag(1,-\alpha(u),-\beta(u))$. Then by Theorem \ref{thm:FormalNormalForm} and Corollary \ref{cor:CkNormalFOrm} we are guaranteed, in a neighbourhood of $ (x,y,z,u) = 0 $, a smooth transformation $ \phi $ conjugating $ X $ to a vector field 
			\[ \tilde{X} = X_0 + \sum_{d \geq 1} \tilde{X}_d, \]
			with $ \tilde{X}_d \in C^\infty \Res_d $. From Lemma \ref{lem:resCondition} each vector field in $ \Res_d $ consists of linear combinations of resonant monomial vector fields,
			\begin{align*}
				x^{n_1}y^{n_2}z^{n_3} \partial_{x} &\text{ such that } n_1 - \alpha(0) n_2 - \beta(0) n_3 - 1 = 0 ,\\
				x^{n_1}y^{n_2}z^{n_3} \partial_{y} &\text{ such that } n_1 - \alpha(0) n_2 - \beta(0) n_3  + \alpha(0) = 0, \\
				x^{n_1}y^{n_2}z^{n_3} \partial_{z} &\text{ such that } n_1 - \alpha(0) n_2 - \beta(0) n_3  + \beta(0) = 0 ,\\
				x^{n_1}y^{n_2}z^{n_3} \partial_{u_i} &\text{ such that } n_1 - \alpha(0) n_2 - \beta(0) n_3 = 0,
			\end{align*}
			for $ n_1,n_2,n_3 \in \N $ and $ n_1 + n_2 + n_3 \geq 2 $. Having a complete description of these resonant monomials will give the normal form. We derive the resonant monomials only for the $ y $ component as the other components follow almost identically.
			
			If $ \alpha(0) = \frac{p_1}{q_1} \in \Q,\ \beta(0) = \frac{p_2}{q_2} \in \Q $, $ \frac{\alpha(0)}{\beta(0)} \nin \N $ with both $ p_1,q_1 $ and $ p_2,q_2 $ co-prime, then a solution to $ n_1 - \alpha(0) n_2 - \beta(0) n_3  + \alpha(0) = 0 $ is given by 
			\[ n_1 = k_1 p_1 + k_2 p_2,\quad n_2 = 1 + k_1 q_1,\quad n_3 = k_2 q_2, \] 
			for $ k_1, k_2 \in \N $ with $ k_1 + k_2 \geq 1 $. This produces the monomial of the form $ y (x^{p_1} y^{q_1})^{k_1} (x^{p_2}z^{q_2})^{k_2} \partial_y = y U_y^{q_1 k_1} U_z^{q_2 k_2} \partial_y $ as desired.
			If $ \alpha(0) \nin \Q $ then we must have $ k_1 = 0 $, hence, the resonant monomial has no $ U_y $ dependence. Similarly if $ \beta(0) \nin \Q $, then $ k_2 = 0 $ and there is no $ U_z $ dependence. These results conclude case 1 of the proposition.
			
			Alternatively, if $ \frac{\alpha(0)}{\beta(0)} \in \N $, then there exists $ m,p,q \in \N $ with $ p,q $ co-prime such that $ \alpha(0) = m \frac{p}{q},\ \beta(0) = \frac{p}{q} $. In such a case, a solution to $ n_1 - \alpha(0) n_2 - \beta(0) n_3  + \alpha(0) = 0 $ is given by 
			\[ n_1 = p k_1,\quad n_2 = 1 + k_2,\quad n_3 = q k_1 - m k_2, \]
			for $ k_1, k_2 \in \Z $ such that $ k_2 \geq -1, 0 \leq q k_1 - m k_2  $. This produces the monomial of the form $ y (x^p z^q)^{k_1} (y z^{-m})^{k_2} \partial_y = U_z^{q k_1-m k_2} U_y^{k_2} \partial_y $ as desired.
			If $ \alpha(0) \nin \Q $ then it must be that $ \beta(0)\nin \Q $. In this instance, $ k_1 = k_2 = -1 $ is the only possible solution. These results conclude case 2 of the proposition.
			
			The function $ K $ is decided from Corollary \ref{cor:CkNormalFOrm}.
			
			Finally, there may be resonant monomials in the $ x $ components of the vector field. Through a smooth time rescaling, all these can be moved from the $ x $ component to the other components. 
		\end{proof}
		\begin{remark}\label{rmk:flattenxyz}
			The difference between the normal forms \eqref{eqn:3DimNFninNu} and \eqref{eqn:3DimNFNu} comes from the additional resonance $ \alpha(0)/\beta(0) \in \N $. Geometrically, this is represented by the fact that $ y = 0, z = 0 $ are invariant in \eqref{eqn:3DimNFninNu} whilst the resonant terms with coefficients $ \alpha_{-1,n_2},\beta_{n_1,-1} $ in \eqref{eqn:3DimNFNu} prevent one from performing a smooth transformation to have these planes invariant. 			
		\end{remark}
		\begin{remark}\label{rmk:planarnormalform}
			The case when $ \NHIM $ is co-dimension 2 is significantly simpler. The normal form is given by restricting to $ z = 0 $ in system \eqref{eqn:3DimNFninNu}. A qualitative depiction of the co-dimension 2 case is given in Figure \ref{fig:1dN}.
		\end{remark}
		\begin{figure}[ht]
			\centering
	\def\svgwidth{0.7\textwidth}
\begingroup%
  \makeatletter%
  \providecommand\color[2][]{%
    \errmessage{(Inkscape) Color is used for the text in Inkscape, but the package 'color.sty' is not loaded}%
    \renewcommand\color[2][]{}%
  }%
  \providecommand\transparent[1]{%
    \errmessage{(Inkscape) Transparency is used (non-zero) for the text in Inkscape, but the package 'transparent.sty' is not loaded}%
    \renewcommand\transparent[1]{}%
  }%
  \providecommand\rotatebox[2]{#2}%
  \newcommand*\fsize{\dimexpr\f@size pt\relax}%
  \newcommand*\lineheight[1]{\fontsize{\fsize}{#1\fsize}\selectfont}%
  \ifx\svgwidth\undefined%
    \setlength{\unitlength}{669bp}%
    \ifx\svgscale\undefined%
      \relax%
    \else%
      \setlength{\unitlength}{\unitlength * \real{\svgscale}}%
    \fi%
  \else%
    \setlength{\unitlength}{\svgwidth}%
  \fi%
  \global\let\svgwidth\undefined%
  \global\let\svgscale\undefined%
  \makeatother%
  \begin{picture}(1,0.71300448)%
    \lineheight{1}%
    \setlength\tabcolsep{0pt}%
    \put(0,0){\includegraphics[width=\unitlength,page=1]{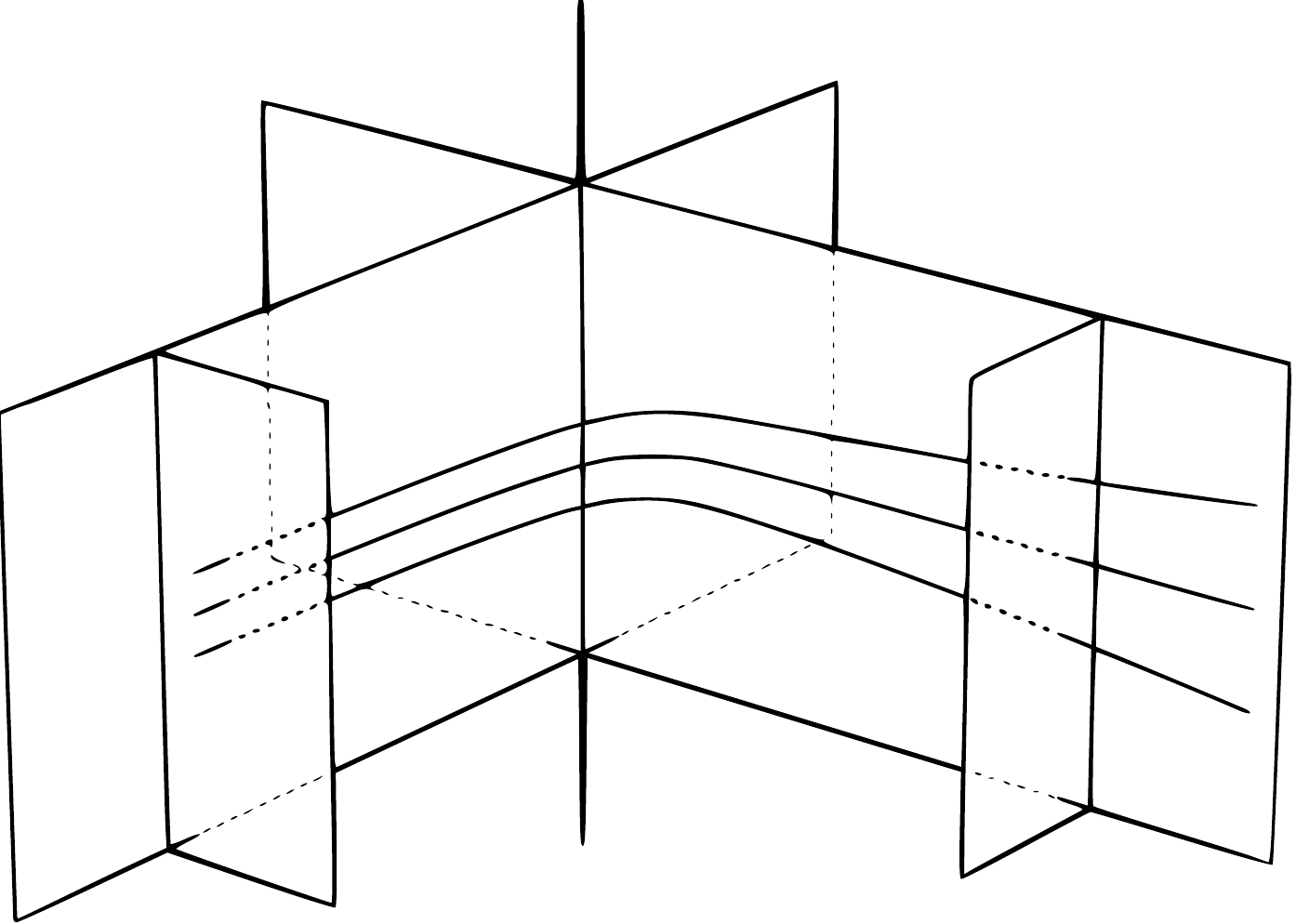}}%
    \put(0.19226834,0.18158677){\color[rgb]{0,0,0}\makebox(0,0)[lt]{\lineheight{2}\smash{\begin{tabular}[t]{l}$\Sigma_y$\end{tabular}}}}%
    \put(0.45610324,0.63021289){\color[rgb]{0,0,0}\makebox(0,0)[lt]{\lineheight{2}\smash{\begin{tabular}[t]{l}$\NHIM$\end{tabular}}}}%
    \put(0.22858569,0.45076244){\color[rgb]{0,0,0}\makebox(0,0)[lt]{\lineheight{2}\smash{\begin{tabular}[t]{l}$W^{sc}(\NHIM)$\end{tabular}}}}%
    \put(0.51271553,0.46464846){\color[rgb]{0,0,0}\makebox(0,0)[lt]{\lineheight{2}\smash{\begin{tabular}[t]{l}$W^{uc}(\NHIM)$\end{tabular}}}}%
    \put(0.76480075,0.18051859){\color[rgb]{0,0,0}\makebox(0,0)[lt]{\lineheight{2}\smash{\begin{tabular}[t]{l}$\Sigma_x$\end{tabular}}}}%
    \put(0,0){\includegraphics[width=\unitlength,page=2]{1dN.pdf}}%
    \put(0.48565879,0.29305695){\color[rgb]{0,0,0}\makebox(0,0)[lt]{\lineheight{2}\smash{\begin{tabular}[t]{l}$D$\end{tabular}}}}%
  \end{picture}%
\endgroup%

			\caption{Diagram of the case $ \NHIM $ is co-dimension 2 in $ \R^3 $}
			\label{fig:1dN}
		\end{figure}
		
		The normal form in Proposition \ref{prop:normalForm3DimN} gives a classification of vector fields $ X $ near a manifold of normally hyperbolic saddle singularities $ \NHIM $. Hence, by studying the flow of \eqref{eqn:3DimNFninNu} and  \eqref{eqn:3DimNFNu} we are able to ascertain properties of all flows near these objects. In particular, we seek an understanding of hyperbolic transitions near $ \NHIM $. 
		
		In what follows, we treat the most general case; when \eqref{eqn:3DimNFninNu} and \eqref{eqn:3DimNFNu} can be considered analytic. Hence, $ K $ will be considered $ \infty $. Finite $ K $ is easily recovered by truncating summations at the relevant order.
		
		Consider the section $ \Sigma = \partial\left([0,1] \times[-1,1]^2\right)\times \R^k $ defined in the normal form coordinates of \eqref{eqn:3DimNFninNu} or  \eqref{eqn:3DimNFNu}. A representation of $ \Sigma $ in relation to $ \NHIM $ is given in Figure \ref{fig:hyperbolic3d} for the case $ \mathcal{N} $ is dimension 0 inside $ \R^3 $ and in Figure \ref{fig:1dN} for the case $ \mathcal{N} $ is co-dimension 2.
		
		\begin{figure}[ht]
			\centering
	\def\svgwidth{0.4\textwidth}
\begingroup%
  \makeatletter%
  \providecommand\color[2][]{%
    \errmessage{(Inkscape) Color is used for the text in Inkscape, but the package 'color.sty' is not loaded}%
    \renewcommand\color[2][]{}%
  }%
  \providecommand\transparent[1]{%
    \errmessage{(Inkscape) Transparency is used (non-zero) for the text in Inkscape, but the package 'transparent.sty' is not loaded}%
    \renewcommand\transparent[1]{}%
  }%
  \providecommand\rotatebox[2]{#2}%
  \newcommand*\fsize{\dimexpr\f@size pt\relax}%
  \newcommand*\lineheight[1]{\fontsize{\fsize}{#1\fsize}\selectfont}%
  \ifx\svgwidth\undefined%
    \setlength{\unitlength}{462.42961121bp}%
    \ifx\svgscale\undefined%
      \relax%
    \else%
      \setlength{\unitlength}{\unitlength * \real{\svgscale}}%
    \fi%
  \else%
    \setlength{\unitlength}{\svgwidth}%
  \fi%
  \global\let\svgwidth\undefined%
  \global\let\svgscale\undefined%
  \makeatother%
  \begin{picture}(1,1.55699372)%
    \lineheight{1}%
    \setlength\tabcolsep{0pt}%
    \put(0,0){\includegraphics[width=\unitlength,page=1]{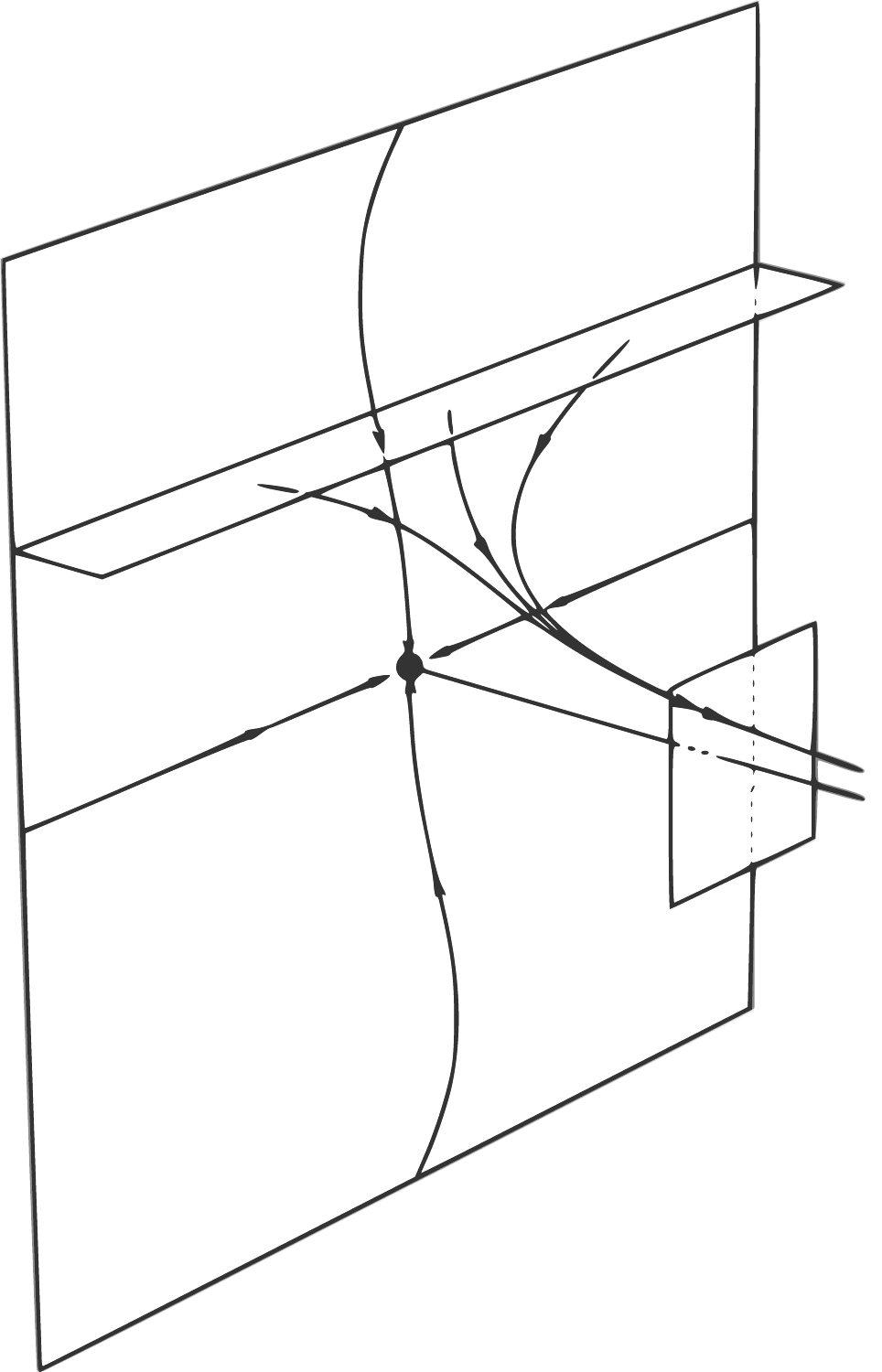}}%
    \put(0.86988149,1.27128549){\color[rgb]{0,0,0}\makebox(0,0)[lt]{\lineheight{2}\smash{\begin{tabular}[t]{l}$ \Sigma_z^+ $\end{tabular}}}}%
    \put(0.89132293,0.86884223){\color[rgb]{0,0,0}\makebox(0,0)[lt]{\lineheight{2}\smash{\begin{tabular}[t]{l}$ \Sigma_x $\end{tabular}}}}%
    \put(0.70329627,0.81441342){\color[rgb]{0,0,0}\makebox(0,0)[lt]{\lineheight{2}\smash{\begin{tabular}[t]{l}$D$\end{tabular}}}}%
    \put(0,0){\includegraphics[width=\unitlength,page=2]{hyperbolic3d.pdf}}%
    \put(0.3859497,0.81725673){\color[rgb]{0,0,0}\makebox(0,0)[lt]{\lineheight{2}\smash{\begin{tabular}[t]{l}$\NHIM$\end{tabular}}}}%
  \end{picture}%
\endgroup%

			\caption{Diagram for when $  \NHIM $ is co-dimension 3 in $ \R^3 $.}
			\label{fig:hyperbolic3d}
		\end{figure}
		
		The interior of $ \Sigma $ is an isolating neighbourhood of $ \NHIM $ in the region $ x \geq 0 $ and is transverse to the centre-stable and centre-unstable manifolds $ x=0 $ and $ y =z = 0 $ respectively. Now, decompose $ \Sigma $ into its various faces,
		\[ \Sigma_x := \Sigma\cap \{x=1 \},\qquad \Sigma_y^\pm := \Sigma\cap\{y=\pm 1 \},\qquad \Sigma_z^\pm := \Sigma \cap\{z=\pm 1 \} \] 
		and note that, due to the fact that $ x=0 $ is the centre-stable manifold, points $ p \in \Sigma_y^\pm \cup\Sigma_z^\pm $ must flow into the interior of $ \Sigma $. Provided that $ p \nin \{x=0\} $, that is $ p $ is not in the centre-stable manifold of $ \NHIM $, we are guaranteed that $ p $ is eventually flowed out of the interior of $ \Sigma $. For $ p $ taken sufficiently close to $ W^{sc}(\NHIM) $, the flow of $ p $ will intersect $ \Sigma_x $. It follows that there is a natural homeomorphism,
		$ \tilde{D}: \Sigma_y^\pm \cup\Sigma_z^\pm\setminus W^{sc}(\NHIM) \to \Sigma_x. $ Moreover, $ \tilde{D} $ admits an extension to a continuous map
		\[ D: \Sigma_y^\pm \cup\Sigma_z^\pm \to \Sigma_x^\pm. \]
		The primary achievement of this section is to obtain an explicit asymptotic series of $ D $ near $ x = 0 $. 
		
		Note that the choice of section $ \Sigma $ is arbitrary. However, the transition for any other choice of section, provided it is transverse to both the stable and unstable manifolds of $ \NHIM $, can be obtained by simply flowing points on $ \Sigma $ to the new section. This transition is smooth, and thus, does not influence the asymptotic structure of $ D $. 
		
		The particular choice of $ \Sigma $ made in this paper has historical precedent. Due to its relevance to Hilbert's $ 16^{th} $ problem, the case when $ \dot{u} = 0 $ and $ \NHIM $ is co-dimension 2 has been well studied; a review is given in \cite{roussarieBifurcationPlanarVector1998}. As $ \dot{u} = 0 $, this case can be considered as a family of hyperbolic singularities in the plane. In this context $ D $ is referred to as the \textit{Dulac map}. Before proceeding to the general case, it is worth mentioning some properties of the Dulac map in the planar case. 
		
		As per Remark \ref{rmk:planarnormalform}, the normal form for the planar case can be deduced from Proposition \ref{prop:normalForm3DimN} by considering $ u $ a parameter and restricting to $ z = 0 $ in case i). Explicitly, the normal form is
		\begin{equation*}
			\begin{aligned}
				\dot{x}	&= x \\
				\dot{y}	&= - \alpha(u) y + y \sum_{n \geq 1} \alpha_{n}(u) (x^p y^q)^n,
			\end{aligned}
		\end{equation*}
		with $ \alpha(0) = p/q \in \Q $.
		The Dulac map is the transition $ D: \Sigma_y^+ = \{y = 1\} \to \Sigma_x = \{ x = 1\} $. There are two key results known for Dulac maps in the planar case. First, if $ x_0 \in \Sigma_y^+ $ then the Dulac map near $ u = 0 $ is asymptotic to the series,
		\[ D(x_0) \sim x_0^{\alpha(u)}\left( 1 + \sum_{i \geq 1} g_i(u, x_0) x_0^{i p} \right), \]
		where $ g_i(x_0) $ is polynomial in the function,
		\begin{equation*}
			\omega(\alpha_1,x_0) = \begin{cases}
				\frac{x_0^{-\alpha_1} - 1}{\alpha_1}, & \alpha_1 \neq 0 \\
				-\ln x_0	& \alpha_1 =0
			\end{cases},
		\end{equation*}
		and $ \alpha_1(u) := \alpha(u) - \alpha(0) $. This function has been denoted the \textit{Ecalle-Roussarie compensator}. It was first introduced in \cite{roussarieNumberLimitCycles1986} and is detailed in \cite[sec.~5.1]{roussarieBifurcationPlanarVector1998}.
		
		The other key result is due to Mourtada \cite{mourtada1990cyclicite}. Setting $ g(u,x_0) = \sum g_i(u, x_0) x_0^{i p} $ it has been shown that,
		\[ \lim_{x_0 \to 0^+} x_0^n \frac{d^n g}{dx_0^n} = 0, \]
		for all $ n \in \N $ and uniformly in $ u $. Functions that exhibit this behaviour are known as \textit{Mourtada type functions}.
		
		Outside of the planar case little is known. Roussarie and Rousseau \cite{roussarieAlmostPlanarHomoclinic1996} investigated the so called `almost planar case'. They treat a family of hyperbolic saddles in $ \R^3 $ with the specific eigenvalue $ \beta(0) = 1 $ and with $ \alpha(0) \nin \Q $ to avoid resonance conditions of Proposition \ref{prop:normalForm3DimN}. In the framework of this paper this case corresponds to an $ \NHIM $ of co-dimension 3 and with $ u $ a parameter, that is, $ \dot{u} = 0 $. They explicitly computed the asymptotic structure of the Dulac map and showed it shares properties with the planar case, namely, its components are Mourtada type functions, and the asymptotic series again contains these $ \omega $ functions. However, by assuming the non-resonance conditions, in particular the case $ \alpha(0)/\beta(0) \in \N $, they did not investigate a crucial difference between the planar case and the co-dimension 3 case.
		
		To see this, take $ \alpha(0),\beta(0) \nin \Q $. From Proposition \ref{prop:normalForm3DimN} the normal form is simply,
		\begin{equation}\label{eqn:basicLinearNormalForm}
			\begin{aligned}
				\dot{x}	&= x \\
				\dot{y}	&= -\alpha(u) y + \alpha_{-1,0}(u) z^m \\
				\dot{z}	&= -\beta(u) z \\
				\dot{u}	&= 0
			\end{aligned}
		\end{equation}
		with $ \alpha_{-1,0}(u) = 0 $ if $ \alpha(0)/\beta(0) \nin \N $. Let $ (x_0,y_0,z_0,u_0)\in\Sigma_y^{\pm} \cup \Sigma_z^{\pm}  $ with $ y_0= \pm 1, z_0 = \pm 1 $ on $ \Sigma_y^{\pm} \cup \Sigma_z^{\pm} $ respectively and take $ (y_1,z_1,u_1) \in \Sigma_x $. Then system \eqref{eqn:basicLinearNormalForm} can be integrated to yield,
		\begin{equation}\label{eqn:nonResAsymptotics}
			\begin{aligned}
				t 	&= -\ln x_0\\
				y_1	&= x_0^{\alpha(u_0)}\left(y_0 + \alpha_{-1,0}(u) z_0^m \omega(\gamma_1(u_0),x_0)\right)\\
				z_1 &= x_0^{\beta(u_0)} z_0 \\
				u_1 &= u_0
			\end{aligned}
		\end{equation}
		with $ \gamma_1(u_0) = \alpha(u_0) - m \beta(u_0) $.
		
		The introduction of the term $ \omega(\gamma_1,x_0) $ due to the resonance $ \alpha(0)/\beta(0) $ prevents the Dulac map from having the same properties as in the planar case. However, for the case $ \dot{u} = 0 $, Bonckaert and Naudot \cite{bonckaertAsymptoticPropertiesDulac2001} were able to show, even in the resonant case, that the Dulac map will always have the form \eqref{eqn:nonResAsymptotics} to leading order. Specifically they showed, for $ D:\Sigma_z^+ \to \Sigma_x $,
		\begin{equation}
			\begin{aligned}
				y_1	&= x_0^{\alpha(u_0)}\left(y_0 + \alpha_{-1,0}(u) \omega(\gamma_1(u_0),x_0) + f(x_0,y_0) \right)\\
				z_1 &= x_0^{\beta(u_0)} (1 + g(x_0,y_0)),
			\end{aligned}
		\end{equation}
		with $ f,g $ functions of Mourtada type. No investigation was made to show the asymptotic structures of $ f,g $ or the case when $ \dot{u} \neq 0 $.
		
		In the remainder of the section we treat each of case i) and ii) from Proposition \ref{prop:normalForm3DimN} in the general case with $ \dot{u} \neq 0 $. The structure of $ f,g $ will be given in Theorem \ref{thm:asymStructureofDNinN} and Theorem \ref{thm:asymptoticStructureOfDN}. The approach taken in the proof of each theorem depends on whether the normal form \eqref{eqn:3DimNFninNu} or \eqref{eqn:3DimNFNu} is considered. The two approaches are similar in concept, but differ in some details.  
		
		\subsection{Case 1: $ \alpha(0) / \beta(0) \nin \N $}\label{sec:ninN}
			
			We proceed by first considering the case $ \alpha(0) / \beta(0) \nin \N $ but $ \alpha(0) = \frac{p_1}{q_1}  $ and $ \beta(0) = \frac{p_2}{q_2} $ with $ p_1, q_1 $ and $ p_2, q_2 $ pairs of co-prime positive integers. The normal form is given by \eqref{eqn:3DimNFninNu}.
			
			Introduce as coordinates \[ U_y = x^{p_1/q_1} y,\qquad U_z = x^{p_2/q_2} z \] and let 
			\[ \alpha(u_0) = \frac{p_1}{q_1} +  \alpha_1(u), \quad \beta(u_0) = \frac{p_2}{q_2} +  \beta_1(u), \]
			where $ \alpha_1,\beta_1 $ are $ O(u) $. Under this coordinate transform the normal form \eqref{eqn:3DimNFninNu} is brought into the form,
			\begin{equation}\label{eqn:c1variationSystemu}
				\begin{aligned}
					\dot{x}	&=  x \\
					\dot{U}_y	&= -\alpha_1(u) U_y + U_y \sum \alpha_{n_1,n_2}(u) U_y^{q_1 n_1} U_z^{q_2 n_2} \\
					\dot{U}_z	&= -\beta_1(u) U_z + U_z \sum \beta_{n_1,n_2}(u) U_y^{q_1 n_1} U_z^{q_2 n_2} \\
					\dot{u}	&= \sum \delta_{n_1,n_2}(u)  U_y^{q_1 n_1} U_z^{q_2 n_2}
				\end{aligned}
			\end{equation}
			The introduction of these coordinates brings the centre-stable manifold $ x = 0 $ to the invariant manifold $ U_y=U_z=0 $. 
			
			We follow \cite{roussarieBifurcationPlanarVector1998} by considering variations of the solutions on $ U_{y}=U_z=0, u = u_0 $. More explicitly, we consider a variation of each orbit $ (U_{y},U_z,u) = (0,0,u_0) $ by a small displacement in $ U_{y}, U_z $ denoted by $ U_{y0}, U_{z0} $ respectively. This variation can be written as a power series of the form,
			\begin{equation}\label{eqn:c1variationsu}
				\begin{aligned}
					U_{y}(U_{y0},U_{z0},u_0; t) 	&= U_{y}^{(1)}(u_0,t) U_{y0} + U_{y0}\sum U_{y}^{(n_1,n_2)}(u_0, t)  U_{y0}^{q_1 n_1} U_{z0}^{q_2 n_2} \\
					U_{z}(U_{y0},U_{z0},u_0; t) 	&= U_{z}^{(1)}(u_0,t) U_{z0} + U_{z0}\sum U_{z}^{(n_1,n_2)}(u_0, t) U_{y0}^{q_1 n_1} U_{z0}^{q_2 n_2}, \\
					u(U_{y0},U_{z0},u_0; t)	&= u_0 + \sum u^{(n_1,n_2)}(u_0, t) U_{y0}^{q_1 n_1} U_{z0}^{q_2 n_2}
				\end{aligned}
			\end{equation}
			with,
			\[  U_{y}^{(1)}(0) = U_{z}^{(1)}(0) = 1,\quad  U_{y}^{(n_1,n_2)}(u_0,0) = U_{z}^{(n_1,n_2)}(u_0,0) = u^{(n_1,n_2)}(u_0,0) = 0, \] so that at $ t=0 $, $ (U_{y},U_{z},u) = (U_{y0},U_{z0},u_0) $.
			
			Each of the coefficient functions $ U_{y}^{(n_1,n_2)}, U_z^{(n_1,n_2)}, u^{(n_1,n_2)} $, referred to as the \textit{variation coefficients}, can be computed through the variational equations. These equations are derived by substituting \eqref{eqn:c1variationsu} into system \eqref{eqn:c1variationSystemu} and equating coefficients of $ U_{y0}^{n_1} U_{z0}^{n_2} $. The first order equations are given by,
			\begin{equation*}
				\begin{aligned}
					\frac{d}{d t} U_{y}^{(1)}	&= -\alpha_1(u_0) U_{y}^{(1)}, 	& U_{y}^{(1)}(0) = 1, \\
					\frac{d}{d t} U_z^{(1)}	&= - \beta_1(u_0) U_z^{(1)} 	& U_z^{(1)}(0) = 1,
				\end{aligned}
			\end{equation*}
			Both equations are linear and hence admit explicit solutions,
			\begin{equation}\label{eqn:Y1Z1sol}
					U_{y}^{(1)}	= \me^{- \alpha_1(u_0) t},\quad	U_z^{(1)}	= \me^{- \beta_1(u_0) t}.
			\end{equation}
			
			The higher order variational equations are given for each $ (n_1,n_2) \in \N^2 $ by,
			\begin{equation}
				\begin{aligned}
					\frac{d}{d t} U_y^{(n_1,n_2)}	&= -\alpha_1(u_0) U_y^{(n_1,n_2)} + R_y^{(n_1,n_2)}, 			& U_y^{(n_1,n_2)}(0) = 0, \\
					\frac{d}{d t} U_{z}^{(n_1,n_2)}	&= -\beta_1(u_0) U_{z}^{(n_1,n_2)} + R_z^{(n_1,n_2)}, 	& U_{z}^{(n_1,n_2)}(0) = 0, \\
					\frac{d}{d t} u^{(n_1,n_2)}	&= R_u^{(n_1,n_2)},		& u^{(n_1,n_2)}(0) = 0,
				\end{aligned}
			\end{equation}
			with $ R_y^{(n_1,n_2)}, R_z^{(n_1,n_2)}, R_u^{(n_1,n_2)} $ polynomial in $ U_{z}^{(\tilde{n}_1,\tilde{n}_2)}, U_y^{(\tilde{n}_1,\tilde{n}_2)}, u^{(\tilde{n}_1,\tilde{n}_2)}  $ for $ \tilde{n}_1+\tilde{n}_2 < n_1 + n_2 $. The equations are linear, thus admit solutions,
			\begin{equation}\label{eqn:HigherVarSols}
				\begin{aligned}
					U_y^{(n_1,n_2)}	&= \me^{-\alpha_1(u_0) t} \int_0^{t} \me^{\alpha_1(u_0) \tau} R_y^{(n_1,n_2)}(\tau) d\tau \\
					U_{z}^{(n_1,n_2)}	&= \me^{-\beta_1(u_0) t} \int_0^{t} \me^{\beta_1(u_0) \tau} R_z^{(n_1,n_2)}(\tau) d\tau \\
					u^{(n_1,n_2)}	&= \int_0^t R_u^{(n_1,n_2)}(\tau) d\tau.
				\end{aligned}
			\end{equation}
			
			A more precise form of the variation coefficients can be given. Take $ \beta \in \R $ and similar to the works on bifurcation theory, for instance \cite{roussarieBifurcationPlanarVector1998}, introduce the function
			\begin{equation}\label{eqn:Omegadef}
				\Omega(\beta_1,t) := \int_{0}^{t} \me^{\beta_1 \tau} d\tau = \begin{cases}
					\frac{\me^{\beta_1 t} - 1}{\beta_1},	& \beta_1 \neq 0,\\
					t									& \beta_1 = 0.
				\end{cases}
			\end{equation}
			Note that $ \lim_{\beta_1 \to 0} \Omega(\beta_1,t) = \Omega(0,t)  $ so that $ \Omega(\beta_1,t) $ can be considered as a family of smooth functions continuous in $ \beta_1 $. 
			
			\begin{definition}~
				
				\begin{enumerate}
					\item Denote by $ \mathcal{O} $ the ring of functions smooth in $ u_0 $ in a neighbourhood of $ 0 $ and rational in $ \alpha,\beta_1\in\R $. 
					\item Denote by $ \mathcal{R}_{\alpha_1,\beta_1} $ the polynomial ring over $ \mathcal{O} $ with indeterminates $ \Omega(\pm\alpha_1,t), \Omega(\pm\beta_1,t), t $. That is,
					\[ \mathcal{R}_{\alpha_1,\beta_1} := \mathcal{O}\left[\Omega(\pm\alpha_1,t),\Omega(\pm\beta_1,t),t\right] . \]
					\item Define the subring $ \bar{\mathcal{R}}_{\alpha_1,\beta_1} $ of elements $ P(\alpha_1,\beta_1; t) \in \mathcal{R}_{\alpha_1,\beta_1} $ such that 
					\[ \lim_{\alpha_1,\beta_1\to 0} P(\alpha_1,\beta_1;t) =: P(0,0; t) \text{ exists,} \]
				\end{enumerate}
			\end{definition}
			
			For example, $ \alpha_1^{-1} \Omega(\alpha_1,t) $ is in $ \mathcal{R}_{\alpha_1,\beta_1} $ but not in $ \bar{\mathcal{R}}_{\alpha_1,\beta_1} $, whilst $ \alpha_1^{-1}(\Omega(\alpha_1,t) - t ) $ is in both.
			
			The following lemmas give essential properties of $ \mathcal{R}_{\alpha_1,\beta_1} $.
			\begin{lemma}\label{lem:closedUnderIntDiff}
				$ \mathcal{R}_{\alpha_1,\beta_1}, \bar{\mathcal{R}}_{\alpha_1,\beta_1} $ are closed under the operators,
				\[ I_t(P) := \int_0^t P(\alpha_1,\beta_1; \tau) d\tau, \qquad D_t(P) := \frac{d}{dt} P(\alpha_1,\beta_1;\tau). \] Moreover $ I_t : \bar{\mathcal{R}}^d_{\alpha_1,\beta_1} \to \bar{\mathcal{R}}^{d+1}_{\alpha_1,\beta_1} $ and $ D_t: \bar{\mathcal{R}}^{d+1}_{\alpha_1,\beta_1} \to \bar{\mathcal{R}}^{d}_{\alpha_1,\beta_1} $.
			\end{lemma}
			\begin{proof}
				If the result can be shown for $ \mathcal{R} $ then by the dominated convergence theorem it is automatically guaranteed for $ \bar{\mathcal{R}} $.  
				
				From the definition of $ \Omega $ in \eqref{eqn:Omegadef} one can see that
        any function $ P \in \mathcal{R}_{\alpha_1,\beta_1} $ can be
        written as a linear combination of functions of the form \[ t^j\me^{(n_1
            \alpha_1 + n_2 \beta_1)t}, \] for some $ j,n_1,n_2 \in \Z $. Through
        the linearity of the integral, it follows $ I_t(P) $ will be a linear
        combination of integrals \[ K_{j} := \int_0^t \tau^j \me^{(n_1 \alpha_1
            + n_2 \beta_1)\tau} d\tau. \] Each of these integrals has the
        recurrence formula
				\[ K_{j} = \frac{1}{n_1 \alpha_1 + n_2 \beta_1} t^{j}\me^{(n_1 \alpha_1 + n_2 \beta_1)t} - \frac{j}{n_1 \alpha_1 + n_2 \beta_1} K_{j-1}. \]
				The recurrence formula, together with the fact that $ \me^{n_1\alpha_1 t} = (1+ \alpha_1 \Omega(\alpha_1,t))^{n_1} $, gives closure of $ \mathcal{R}_{\alpha_1,\beta_1} $ under integration.
				
				Similarly, \[ D_t(t^j\me^{(n_1 \alpha_1 + n_2 \beta_1)t}) = (j t^{j-1} + (n_1 \alpha_1 + n_2 \beta_1) t^j ) \me^{(n_1 \alpha_1 + n_2 \beta_1)t}. \]
				Hence, the closure under $ D_t $ is guaranteed.
			\end{proof}
			\begin{lemma}\label{lem:polynomialR}
				Let $ P(\alpha_1,\beta_1;t) \in \bar{\mathcal{R}}_{\alpha_1,\beta_1} $. Then $ P(0,0;t) $ is polynomial in $ t $.
			\end{lemma}
			\begin{proof}
				$ P(\alpha_1,\beta_1;t) $ can be written as a linear combination of functions of the form, $ f(\alpha_1,\beta_1)t^j \me^{(n_1 \alpha_1 + n_2 \beta_1)t} $ where $ f $ is a rational function. As $ f $ is rational then by definition there exists $ p,q $ polynomial in $ \alpha_1, \beta_1 $ with $ f(\alpha_1,\beta_1) = p(\alpha_1,\beta_1) / q(\alpha_1,\beta_1) $. Let $ d_p,d_q $ be the degree of $p, q $ respectively. If $ d_p - d_q > 0 $ then $ \lim_{\alpha_1,\beta_1 \to 0} f(\alpha_1,\beta_1) = 0 $.
				
				Now, if $ P\in \bar{\mathcal{R}}_{\alpha_1,\beta_1} $ we must have $  \lim_{\alpha_1,\beta_1\to 0} \frac{d^k}{dt^k} P(\alpha_1,\beta_1;t) = \frac{d^k}{dt^k} P(0,0;t) $. The derivative $ d/dt  f(\alpha_1,\beta_1)t^j  $ gives the function $(j t^{j-1} + (n_1 \alpha_1 + n_2 \beta_1) t^j ) \me^{(n_1 \alpha_1 + n_2 \beta_1)t} $ which is the sum of a function of one degree less in $ t $ and a function with coefficient $ (n_1 \alpha_1 + n_2 \beta_1) p(\alpha_1,\beta_1)/q(\alpha_1,\beta_1) $. The coefficient is again rational with sum of degrees $ d_p - d_q +1 $. Hence, there exists $ k < \infty $ such that, for all $ \tilde{k} > k $, $ \frac{d^{\tilde{k}}}{dt^{\tilde{k}}} P(\alpha_1,\beta_1;t) $ contains only terms with coefficients $ f = p/q $ with sum of degrees $ d_p - d_q > 0 $. Taking the limit $ \alpha_1,\beta_1 \to 0 $ gives $ \frac{d^{\tilde{k}}}{d t^{\tilde{k}}} P(0,0; t) = 0 $ for all $ \tilde{k} > k $. It follows that $ P (0,0;t) $ is polynomial in $ t $.
			\end{proof}

			With the definition of $ \mathcal{R}_{\alpha_1,\beta_1} $ given and the preceding lemmas, we have the following proposition on the form of the variation coefficients.
			\begin{proposition}\label{prop:varStructureNinN}
				For all $ (n_1,n_2) \in \N^2 $ there exists functions \[ \tilde{U}_y^{(n_1,n_2)}(u_0,t),\tilde{U}_z^{(n_1,n_2)}(u_0,t), \tilde{u}_i^{(n_1,n_2)}(u_0,t) \in \bar{\mathcal{R}}_{\alpha_1,\beta_1},\qquad i = 1,\dots,k, \] 
				such that,
				\begin{align*}
					U_y^{(n_1,n_2)}(u_0; t) &= \me^{-\alpha_1(u_0) t} \tilde{U}_y^{(n_1,n_2)}(u_0,t) \\
					U_z^{(n_1,n_2)}(u_0; t) &= \me^{-\beta_1(u_0) t} \tilde{U}_z^{(n_1,n_2)}(u_0,t) \\
					u^{(n_1,n_2)}(u_
					0; t) &= \tilde{u}^{(n_1,n_2)}(t)
				\end{align*}
				with $ \tilde{u}^{(n_1,n_2)} := \left(u_1^{(n_1,n_2)},\dots,\tilde{u}_k^{(n_1,n_2)}\right) $. Moreover:
				\begin{enumerate}[i)]
					\item Each $ \tilde{U}_y^{(n_1,n_2)},\tilde{U_z}^{(n_1,n_2)},\tilde{u}^{(n_1,n_2)} $ is polynomial in $ \alpha_{\tilde{n}_1,\tilde{n}_2},\beta_{\tilde{n}_1,\tilde{n}_2},\delta_{\tilde{n}_1,\tilde{n}_2} $ for $ \tilde{n}_1+\tilde{n}_2 \leq n_1+n_2 $.
					\item If $ \alpha_{n_1,n_2} $ (resp. $ \beta_{n_1,n_2},\delta_{n_1,n_2}^i $ ) vanish for $ n_1 + n_2 \leq n\in \N $ then $ U_y^{(n_1,n_2)}(t)  $ (resp. $ U_z^{(n_1,n_2)}, U_{u_i}^{(n_1,n_2)} $) vanish for $ n_1 + n_2 \leq n.$
				\end{enumerate}
			\end{proposition}
			\begin{proof}
				The proposition will be proved by induction on $ k = n_1 + n_2 $. From \eqref{eqn:Y1Z1sol} it is known that 
				\[ U_y^{(0,0)} = U_y^{(1)} = \me^{-\alpha_1(u_0) t }\cdot 1,\quad U_z^{(0,0)}=U_z^{(1)} = \me^{-\beta_1(u_0) t} \cdot 1,\quad u^{(0,0)} = u_0.\]
				As $ 1 $ and each component of $ u_0 $ are elements of $ \bar{\mathcal{R}}_{\alpha_1,\beta_1} $ the result is true for $ k = 0 $.
				
				Now assume true for all $ n_1,n_2 \in N $ such that $ n_1 + n_2 < k $. Take any $ n_1,n_2 \in \N $ with $ n_1 + n_2 = k $ and let $ K $ represent each of $ U_y,U_z,u $. It was shown that each $ K^{(n_1,n_2)} $ are given by the solutions to the variational equations computed in \eqref{eqn:HigherVarSols}. As remarked before \eqref{eqn:HigherVarSols}, each $ R_K^{(n_1,n_2)} $ is a polynomial in $ K^{(\tilde{n}_1,\tilde{n}_1)} $ for $ \tilde{n}_1 + \tilde{n}_2 \leq n_1 + n_2 = k $, and as such, if each $ K^{(\tilde{n}_1,\tilde{n}_1)} \in \bar{\mathcal{R}}_{\alpha_1,\beta_1} $ by assumption, then $ R_K^{(n_1,n_2)} \in \bar{\mathcal{R}}_{\alpha_1,\beta_1} $. Furthermore, $ \me^{\kappa t} = (1+\kappa \Omega(\kappa, t)) $ for $ \kappa = \alpha_1,\beta_1,-\alpha_1,-\beta_1 $. Hence, $ \me^{\alpha_1(u_0) t} R_y^{(n_1,n_2)} , \me^{\beta_1(u_0) t} R_z^{(n_1,n_2)}, R_u^{(n_1,n_2)} $ are all elements of $ \bar{\mathcal{R}}_{\alpha_1,\beta_1} $.
				
				By Lemma \ref{lem:closedUnderIntDiff} $ \bar{\mathcal{R}}_{\alpha_1,\beta_1} $ is closed under integration. Thus we can set 
				\begin{align*}	
					\tilde{U}_y^{(n_1,n_2)} &:= \int_0^t (1+\alpha_1 \Omega(\alpha_1,\tau))R_y^{(n_1,n_2)} d\tau, \\
					\tilde{U_z}^{(n_1,n_2)} &:= \int_0^t (1+\beta_1 \Omega(\beta_1,\tau))R_z^{(n_1,n_2)} d\tau, \\
					\tilde{u}^{(n_1,n_2)} &:= \int_0^t R_u^{(n_1,n_2)} d\tau,
				\end{align*}
				to conclude the proposition. 
				
				The fact that $ \tilde{U}_y^{(n_1,n_2)},\tilde{U_z}^{(n_1,n_2)}(t), \tilde{u}_i^{(n_1,n_2)}(t)  $ are polynomial in $ \alpha_{n_1,n_2}, \beta_{n_1,n_2}, \delta_{n_1,n_2} $ is a consequence of the polynomial nature of $ R_y,R_z,R_u $. Property ii) follows from the fact that the remainder terms vanish if there are no lower order nonlinear terms in \eqref{eqn:c1variationSystemu}.
			\end{proof}
			
			At last we return to the Dulac map $ D $. The time to go from $ \Sigma_y^\pm \cup \Sigma_z^\pm $ to $ \Sigma_x $ can be computed from $ \dot{x} = x $ as simply $ t = -\ln x_0 $. The transition maps can be derived from the solution to the variational equations using at $ t = 0 $, $ (U_{y0},U_{z0}) = (x_0^{p_1/q_1} y_0, x_0^{p_2/q_2} z_0) $ and at $ t = -\ln x_0 $, $ (U_y,U_z,u) = (y_1,z_1,u_1) $. That is,
			\begin{equation}\label{eqn:DFormNinN}
				\begin{aligned}
					y_1 	&= U_y(x_0^{p_1/q_1} y_0,x_0^{p_2/q_2} z_0,u_0,-\ln x_0), \\
					z_1 	&= U_z(x_0^{p_1/q_1} y_0,x_0^{p_2/q_2} z_0,u_0,-\ln x_0), \\
					u_1		&= u(x_0^{p_1/q_1} y_0,x_0^{p_2/q_2} z_0,u_0,-\ln x_0), 
				\end{aligned}
			\end{equation} 
			with $ y_0 = \pm 1 $, $ z_0 = \pm 1 $ when mapping from $ \Sigma_y^\pm,\Sigma_z^\pm $ respectively.
			
			Define the Ecalle-Roussarie compensator by,
			\begin{equation}
				\begin{cases}
					\omega(\alpha_1,x) = \displaystyle\frac{x^{-\alpha_1} - 1}{\alpha_1} & \alpha_1 \neq 0 \\
					\omega(0,x) = -\ln x
				\end{cases}.
			\end{equation}
			The function $ \omega $ is related to $ \Omega $ by 
			\[ \omega(\alpha_1,x) = \Omega(\alpha_1,-\ln x). \] 
			By taking $ t = -\ln x $ in the definition of $ \mathcal{R}_{\alpha_1,\beta_1}, \bar{\mathcal{R}}_{\alpha_1,\beta_1} $ there are induced rings $ \mathcal{R}_{\alpha_1,\beta_1}^\omega,\bar{\mathcal{R}}^\omega_{\alpha_1,\beta_1} $. 
			
			At last, we have the following theorem on the asymptotic structure of the Dulac map.			
			\begin{thm}\label{thm:asymStructureofDNinN}
				Suppose that $ \alpha(0)/\beta(0) \nin \N $. Then the Dulac map $ D $ is asymptotic to the series
				\begin{equation}
					\begin{aligned}
						y_1	&\sim x_0^{\alpha(u_0)} y_0 \left( 1  + \sum_{n_1+n_2 \geq 1} \bar{U_y}^{(n_1,n_2)}(u_0; x_0) (x_0^{p_1} y_0^{q_1})^{n_1}( x_0^{p_2} z_0^{ q_2})^{n_2} \right) \\
						z_1	&\sim x_0^{\beta(u_0)} z_0 \left(  1 + \sum_{n_1 + n_2 \geq 1} \bar{U_{z}}^{(n_1,n_2)}(u_0; x_0) (x_0^{p_1} y_0^{q_1})^{n_1}( x_0^{p_2} z_0^{ q_2})^{n_2} \right)  \\
						u_1	&\sim u_0 + \sum_{n_1 + n_2 \geq 1} \bar{u}^{(n_1,n_2)}(u_0; x_0) (x_0^{p_1} y_0^{q_1})^{n_1}( x_0^{p_2} z_0^{ q_2})^{n_2} 
					\end{aligned}
				\end{equation}
				with $ y_0 = \pm 1 $, $ z_0 = \pm 1 $ when mapping from $ \Sigma_y^\pm,\Sigma_z^\pm $ respectively. Each coefficient $ K^{(n_1,n_2)} = \bar{U_y}^{(n_1,n_2)}, \bar{U_{z}}^{(n_1,n_2)} $ or $\bar{u}_i^{(n_1,n_2)}$, $ i=1,\dots,k, $ has the properties:
				\begin{enumerate}[i)]
					\item $ K^{(n_1,n_2)}\in \bar{\mathcal{R}}^\omega_{\alpha_1,\beta_1} $. \\
					\item If $ \alpha(u_0), \beta(u_0) $ are constant then $ K^{(n_1,n_2)} $ is polynomial in $ \ln x_0 $.
					\item $ K^{(n_1,n_2)} $ is polynomial in $ \alpha_{\tilde{n}_1,\tilde{n}_2},\beta_{\tilde{n}_1,\tilde{n}_2},\delta_{\tilde{n}_1,\tilde{n}_2} $ for $ \tilde{n}_1+\tilde{n}_2 \leq n_1+n_2 $ with vanishing constant term.
					\item If $ \alpha_{n_1,n_2} $ (resp. $ \beta_{n_1,n_2},\delta_{n_1,n_2}^i $ ) vanish for $ n_1 + n_2 \leq n\in \N $ then $ \bar{U}_y^{(n_1,n_2)}(t)  $ (resp. $ \bar{U}_z^{(n_1,n_2)}, \bar{U}_{u_i}^{(n_1,n_2)} $) vanish for $ n_1 + n_2 \leq n.$
				\end{enumerate}
			\end{thm} 
			\begin{proof}
				The proof is primarily a consequence of Proposition \ref{prop:varStructureNinN} and the form of $ D $ given in \eqref{eqn:DFormNinN}. The explicit computation is given for $ y_1 $ with the $ z_1,u_1 $ following analogously. It is given that,
				\[ 	y_1 	= U_y(x_0^{p_1/q_1} y_0,x_0^{p_2/q_2} z_0, u_0, -\ln x_0). \]
				An asymptotic expansion for $ U_y $ is given by the variation of $ U_y $ in \eqref{eqn:c1variationsu}, that is, 
				\[ U_y(U_{y0},U_{z0},u_0; t) 	\sim U_{y}^{(1)}(u_0,t) U_{y0} + U_{y0}\sum U_{y}^{(n_1,n_2)}(u_0, t)  U_{y0}^{q_1 n_1} U_{z0}^{q_2 n_2}. \]
				Then, from Proposition \ref{prop:varStructureNinN} each of the variational coefficients $ U_{y}^{(n_1,n_2)}(u_0,t) $ has the structure,
				\[ U_{y}^{(n_1,n_2)}(u_0, t) = \me^{-\alpha_1(u_0) t} \tilde{U_{y}}^{(n_1,n_2)}(t) \]
				with $ \tilde{U_{y}}^{(n_1,n_2)}(t) \in \mathcal{R}_{\alpha_1,\beta_1}. $ By substituting $ t = - \ln x_0 $, it follows,
				\[  U_{y}^{(n_1,n_2)}(u_0, -\ln x_0) =  x_0^{\alpha_1(u_0)} \hat{U_{y}}^{(n_1,n_2)}(x_0), \]
				for some $ \hat{U_{y}}^{(n_1,n_2)}(x_0)\in \mathcal{R}_{\alpha_1,\beta_1}^\omega $. Hence,
				\begin{align*}
					y_1 &\sim x_0^{\alpha_1(u_0)}U_{y0} + U_{y0}\sum x_0^{\alpha_1(u_0)}\hat{U_{y}}^{(n_1,n_2)}(u_0; x_0)  U_{y0}^{q_1 n_1} U_{z0}^{q_2 n_2} \\
					 &= x_0^{\alpha_1(u_0)}x_0^{p_1/q_1} y_0+ x_0^{p_1/q_1} y_0\sum x_0^{\alpha_1(u_0)}\hat{U_{y}}^{(n_1,n_2)}(u_0; x_0) (x_0^{p_1/q_1} y_0)^{q_1 n_1} (x_0^{ p_2/q_2} z_0 )^{n_2} \\
					 &= x_0^{p_1/q_1 + \alpha_1(u_0) } y_0\left( 1 + \sum \hat{U_{y}}^{(n_1,n_2)}(u_0; x_0)  y_0^{n_1 q_1} z_0^{n_2 q_2} x_0^{n_1 p_1+ n_2 p_2} \right).
				\end{align*}
				The desired asymptotic form of the $ y_1 $ component of $ D $ follows. 
							
				Properties i), iii) and iv) follow immediately from Proposition \ref{prop:varStructureNinN}. If $ \alpha(u_0), \beta(u_0) $ are constant then $ \alpha_1(u_0)=\beta_1(u_0)=0 $. The form can be computed by taking $$ \lim_{\alpha_1,\beta_1 \to 0} \hat{U_{y}}^{(n_1,n_2)}(u_0; x_0). $$ As $ \hat{U_{y}}^{(n_1,n_2)} \in \bar{\mathcal{R}}_{\alpha_1,\beta_1} $ then Lemma \ref{lem:polynomialR} gives property ii).
			\end{proof}
		
			\begin{remark}
				Setting $ z_0 = 0, y_0 = 1 $ gives the Dulac map of a co-dimension 2 manifold of normally hyperbolic saddle singularities. If it is further assumed that $ u $ is merely a parameter, that is $ \dot{u} =0 $, then Theorem \ref{thm:asymStructureofDNinN} gives the asymptotic structure of the transition near a family of planar hyperbolic saddles. This result agrees with \cite{roussarieBifurcationPlanarVector1998}.
			\end{remark}
		
		\subsection{Case 2: $ \alpha(0)/\beta(0) \in \N $}
			In this section we treat the case $ \alpha(0)/\beta(0) \in \N $. The general approach is the same as in the previous section, however some minor care needs to be taken when dealing with the coefficients $ \alpha_{-1,n_2},\beta_{n_1,-1} $ in the normal form \eqref{eqn:3DimNFNu}.
			
			To make summation symbols less cumbersome, define the following subsets of $ \N^2 $,
			\begin{equation}
				\begin{aligned}
					N_1	&:= \set{(n_1,n_2) \in \N^2}{n_1 \geq -1, q n_2 - m n_1 \geq 0, (n_1,n_2)\neq 0} \\
					N_2	&:= \set{(n_1,n_2) \in \N^2}{n_1 \geq 0, q n_2 - m n_1 \geq -1, (n_1,n_2)\neq 0}\\
					N_3	&:= \set{(n_1,n_2) \in \N^2}{n_1 \geq 0, q n_2 - m n_1 \geq 0, (n_1,n_2)\neq 0}.
				\end{aligned}
			\end{equation}
			
			Then, introduce as coordinates 
			\[ U_y = x^{m p /q} y,\qquad U_z = x^{p/q} z, \]
			and define $ \alpha_1,\beta_1 $ through,
			\[ \alpha(u_0) = m\frac{p}{q} + \alpha_1(u),\quad \beta(u_0) = \frac{p}{q} + \beta_1(u). \]
			In these new coordinates the normal form \eqref{eqn:3DimNFNu} is transformed to the vector field,
			\begin{equation}
				\begin{aligned}
					\dot{x}	&=  x \\
					\dot{U}_y	&= -\alpha_1(u) U_y + U_y \sum_{(n_1,n_2) \in N_1} \alpha_{n_1,n_2}(u) U_y^{n_1} U_z^{q n_2 - m n_1} \\
					\dot{U}_z	&= -\beta_1(u) U_z + U_z \sum_{(n_1,n_2)\in N_2} \beta_{n_1,n_2}(u) U_y^{n_1} U_z^{ q n_2 - m n_1} \\
					\dot{u}	&= \sum_{(n_1,n_2)\in N_3} \delta_{n_1,n_2}(u)  U_y^{ n_1} U_z^{ n_2}
				\end{aligned}
			\end{equation}
			The crucial achievement of the coordinate transform is to decouple $ U_y,U_z,u $ from $ x $. 
			
			The centre-stable manifold $ x = 0 $ has been brought to $ U_z=U_y = 0 $. Similar to Section \ref{sec:ninN}, we consider variations of the solutions $ U_y=U_z=0 $. More explicitly, we consider a variation of the form,
			\begin{equation}\label{eqn:c2variationsu}
				\begin{aligned}
					U_{y}(U_{y0},U_{z0},u_0; t) 	&= U_{y}^{(1)}(u_0,t) U_{y0} + U_{y0}\sum_{(n_1,n_2) \in N_1} U_{y}^{(n_1,n_2)}(u_0, t)  U_{y0}^{n_1} U_{z0}^{ q n_1 - m n_2} \\
					U_{z}(U_{y0},U_{z0},u_0; t) 	&= U_{z}^{(1)}(u_0,t) U_{z0} + U_{z0}\sum_{(n_1,n_2) \in N_2} U_{z}^{(n_1,n_2)}(u_0, t) U_{y0}^{n_1} U_{z0}^{ q n_1 - m n_2}, \\
					u(U_{y0},U_{z0},u_0; t)	&= u_0 + \sum_{(n_1,n_2) \in N_3} u^{(n_1,n_2)}(u_0, t) U_{y0}^{n_1} U_{z0}^{ q n_1 - m n_2}
				\end{aligned}
			\end{equation}
			with,
			\[  U_{y}^{(1)}(0) = U_{z}^{(1)}(0) = 1,\quad  U_{y}^{(n_1,n_2)}(u_0,0) = U_{z}^{(n_1,n_2)}(u_0,0) = u^{(n_1,n_2)}(u_0,0) = 0, \] so that at $ t=0 $, $ (U_{y},U_{z},u) = (U_{y0},U_{z0},u_0) $.
			
			The following proposition gives the structure of the variation coefficients.
			\begin{proposition}\label{prop:varStructureN}
				There exists functions $ \tilde{U}_y^{(n_1,n_2)},\tilde{U}_z^{(n_1,n_2)}, \tilde{u}_i^{(n_1,n_2)} \in \bar{\mathcal{R}}_{\alpha_1,\beta_1} $ such that,
				\begin{align*}
					U_y^{(n_1,n_2)}(u_0; t) &= \me^{-\alpha_1(u_0) t} \tilde{U}_y^{(n_1,n_2)}(t) \\
					U_z^{(n_1,n_2)}(u_0; t) &= \me^{-\beta_1(u_0) t} \tilde{U_z}^{(n_1,n_2)}(t) \\
					u^{(n_1,n_2)}(u_0; t) &= \tilde{u}^{(n_1,n_2)}(t)
				\end{align*}
				with $ \tilde{u}^{(n_1,n_2)} := \left(u_1^{(n_1,n_2)},\dots,\tilde{u}_k^{(n_1,n_2)}\right) $. Moreover:
				\begin{enumerate}[i)]
					\item Each $ \tilde{U}_y^{(n_1,n_2)},\tilde{U_z}^{(n_1,n_2)},\tilde{u}^{(n_1,n_2)} $ is polynomial in $ \alpha_{\tilde{n}_1,\tilde{n}_2},\beta_{\tilde{n}_1,\tilde{n}_2},\delta_{\tilde{n}_1,\tilde{n}_2} $ for $ \tilde{n}_1 + q \tilde{n}_2 - m \tilde{n}_1 \leq n_1 + q n _2 - m n_1 $ with zero constant term. .
					\item If $ \alpha_{n_1,n_2} $ (resp. $ \beta_{n_1,n_2},\delta_{n_1,n_2}^i $ ) vanish for $ n_1 + q n _2 - m n_1 \leq n\in \N $ then $ U_y^{(n_1,n_2)}(t)  $ (resp. $ U_z^{(n_1,n_2)}, U_{u_i}^{(n_1,n_2)} $) vanish for $ n_1 + q n _2 - m n_1 \leq n.$
				\end{enumerate}
			\end{proposition}
			The proof is omitted as it is almost identical to Proposition \ref{prop:varStructureNinN}, namely, using induction on $ n_1,n_2 $ to show that the integral solution to the variational equations gives the desired functions $  \tilde{U}_y^{(n_1,n_2)}(t), \tilde{U}_z^{(n_1,n_2)}(t),	\tilde{u}^{(n_1,n_2)}(t) $.
			
			Returning to the Dulac map, one again computes the time to go from $ \Sigma_y^\pm\cup\Sigma_z^\pm $ to $ \Sigma_x $ as simply $ t = -\ln x_0 $. We have the relation,
			\begin{equation}\label{eqn:y1z1u1N}
				\begin{aligned}
					y_1	&= U_y( x_0^{mp/q}y_0, x_0^{p/q} z_0, u_0,-\ln x_0), \\
					z_1	&= U_z( x_0^{mp/q}y_0, x_0^{p/q} z_0, u_0,-\ln x_0), \\
					u_1 &= u( x_0^{mp/q}y_0, x_0^{p/q} z_0, u_0,-\ln x_0). 
				\end{aligned}
			\end{equation} 
			
			The theorem on the asymptotic structure of the Dulac map follows.
			\begin{thm}\label{thm:asymptoticStructureOfDN}
				Suppose that $ \alpha(0)/\beta(0) \in \N $ and set $ \gamma_1 = \alpha - m \beta $. Then the Dulac map is asymptotic to the series,
				\begin{equation}
					\begin{aligned}
						y_1	&\sim x_0^{\beta(u_0)}\left( y_0 + \alpha_{-1,0}(u_0) z_0^m \omega(\gamma_1,x_0) + y_0 \sum_{(n_1,n_2)\in N_1} \bar{U}_y^{(n_1,n_2)}(u_0;x_0) (x_0^{mp} y_0^q)^{\frac{1}{q} n_1}(x_0^p z_0^q)^{n_2 - \frac{m}{q} n_1} \right) \\
						z_1	&\sim x_0^{\alpha(u_0)}\left( z_0 + z_0 \sum_{(n_1,n_2)\in N_2} \bar{U}_z^{(n_1,n_2)}(u_0;x_0)(x_0^{mp} y_0^q)^{\frac{1}{q} n_1}(x_0^p z_0^q)^{n_2 - \frac{m}{q} n_1} \right) \\
						u_1	&\sim u_0 +  \sum_{(n_1,n_2)\in N_3} \bar{u}^{(n_1,n_2)}(u_0;x_0)(x_0^{mp} y_0^q)^{\frac{1}{q} n_1}(x_0^p z_0^q)^{n_2 - \frac{m}{q} n_1} 
					\end{aligned}
				\end{equation}
				with $ y_0 = \pm 1 $, $ z_0 = \pm 1 $ when mapping from $ \Sigma_y^\pm,\Sigma_z^\pm $ respectively. Each coefficient $ K^{(n_1,n_2)} = \bar{U_y}^{(n_1,n_2)}, \bar{U_{z}}^{(n_1,n_2)} $ or $\bar{u}_i^{(n_1,n_2)}$, $ i=1,\dots,k, $ has the properties:
				\begin{enumerate}[i)]
					\item $ K^{(n_1,n_2)}\in \bar{\mathcal{R}}^\omega_{\alpha_1,\beta_1} $. \\
					\item If $ \alpha(u_0), \beta(u_0) $ are constant then $ K^{(n_1,n_2)} $ is polynomial in $ \ln x_0 $.
					\item $ K^{(n_1,n_2)} $ is polynomial in $ \alpha_{\tilde{n}_1,\tilde{n}_2},\beta_{\tilde{n}_1,\tilde{n}_2},\delta_{\tilde{n}_1,\tilde{n}_2} $ for $ \tilde{n}_1 + q \tilde{n}_2 - m \tilde{n}_1 \leq n_1 + q n _2 - m n_1  $ with zero constant term. 
					\item If $ \alpha_{n_1,n_2} $ (resp. $ \beta_{n_1,n_2},\delta_{n_1,n_2}^i $ ) vanish for $ n_1 + q n _2 - m n_1 \leq n\in \N $ then $ \bar{U}_y^{(n_1,n_2)}(t)  $ (resp. $ \bar{U}_z^{(n_1,n_2)}, \bar{U}_{u_i}^{(n_1,n_2)} $) vanish for $ n_1 + q n _2 - m n_1 \leq n.$
				\end{enumerate}
			\end{thm} 
			\begin{proof}
				The proof is almost identical to the proof of Theorem \ref{thm:asymStructureofDNinN}, namely, using equation \eqref{eqn:y1z1u1N}, Proposition \ref{prop:varStructureN} and substituting $ t = -\ln x_0 $ into the solution to the variational equations to get the asymptotic structure. The only difference is showing the additional $ \alpha_{-1,0} z_0^m \omega(\gamma_1,x_0) $ term in the $ y_1 $ component of the Dulac map $ D $. This comes from the variational coefficient $ U_y^{(-1,0)}(u_0,t) $. The coefficient must solve the variational equation
				\[ \frac{d}{dt} U_y^{(-1,0)}(u_0,t) = -\alpha_1(u_0) U_y^{(-1,0)}(u_0,t) + \alpha_{-1,0}(u_0) U_z^{(1)}(u_0,t).  \]
				By Proposition \ref{prop:varStructureN} it is known that $ U_z^{(1)}(u_0,t) = \me^{-\beta_1(u_0) t} $. It follows that,
				\[ U_y^{(-1,0)}(u_0,t) = \alpha_{-1,0}\Omega(\alpha_1 - m \beta_1, t) = \Omega(\alpha - m \beta,t). \]
				Finally, $ U_y^{(-1,0)} $ is the coefficient of $ U_{z0}^m $ in the $ U_y $ variation. Substituting $ U_{z0} = x_0^{p/q} z_0 $ as per equation \ref{eqn:y1z1u1N} yields the desired term in the asymptotic expansion of $ y_1 $.
			\end{proof}
		
			\begin{remark}
				Due its applicability to problems in celestial mechanics, especially \cite{duignanC83regularisationSimultaneous2020}, it is worth isolating the case when $ \alpha,\beta $ take constant values on $ \NHIM $. In the co-dimension 2 case, one obtains the asymptotic series by setting $ z_0 = 0, y_0 = 1 $ in Theorem \ref{thm:asymStructureofDNinN} and invoking property ii) to get,
				\begin{equation}
					\begin{aligned}
						y_1	&\sim x_0^{\alpha} \left( 1  + \sum_{n \geq 1} \hat{U}_y^{(n)}(u_0; \ln x_0) x_0^{n p} \right) \\
						u_1	&\sim u_0 + \sum_{n \geq 1} \hat{u}^{(n)}(u_0; \ln x_0) x_0^{n p},
					\end{aligned}
				\end{equation}
				for functions $ \hat{U}_y^{(n)}, \hat{u}^{(n)} $ polynomial in $ \ln x_0 $ and smooth in $ u_0 $.
			\end{remark}
	
				It is now evident that the asymptotic structure of the higher dimensional Dulac maps $ D $ share similar properties to the well known planar case. In the planar case the coefficients functions $ g_i(u,x_0) $ are known to be polynomial in the functions $ \omega(\alpha_1, x_0) $. This is mirrored in the present case with each of the coefficients $ K^{(n_1,n_2)} \in \mathcal{R}_{\alpha_1,\beta_1}^\omega $, the ring of polynomials in $ \omega(\pm\alpha_1,x_0),\omega(\pm\beta_1,x_0) $. The Mourtada property of the higher order asymptotic terms, first shown in the case $ \dot{u} = 0 $ in \cite{bonckaertAsymptoticPropertiesDulac2001}, should also be evident. 
        
		\section*{Acknowledgment}
	    The author would like to thank Holger Dullin for all the discussions and
	    constructive criticisms of which have made this paper possible. Thanks must
	    also be given to Robert Roussarie for the many comments that greatly
	    improved an earlier version of this manuscript, and to the reviewers for their careful reading. 
		
	\bibliography{Transitions}
	\bibliographystyle{alpha}

\end{document}